\documentclass[12pt]{article}

\newtheorem{theorem}{Theorem}[section]
\usepackage[square,numbers]{natbib}
\usepackage{amsfonts}
\usepackage{amsmath}
\usepackage{amssymb}
\usepackage{fullpage}
\usepackage{stmaryrd} 
\usepackage{wasysym} 
\usepackage[colors]{optsys}

\multiopssep{setco}{}{\{}{:}{\}}

\usepackage{graphicx}
\usepackage{tikz}
\usepackage{caption}
\usepackage{hyperref}
\usepackage{subcaption}
\usepackage[T1]{fontenc}
\usepackage{tikz}
\usetikzlibrary{positioning}
\usepackage{pgfplots}
\pgfplotsset{compat=1.5.1}
\usepgfplotslibrary{dateplot}
\usepgfplotslibrary{fillbetween}
\usetikzlibrary{patterns}
\usepackage{color}
\colorlet{mygreen}{black!50!green!60}
\colorlet{myblue}{black!30!blue}

\usepackage{float} 
\usepackage{tikz}
\usepackage{tikz-cd}
\usetikzlibrary{patterns}
\usetikzlibrary{shapes.geometric}
\usetikzlibrary{positioning}
\usepackage{dirtytalk}

\newcommand{\RadialProjection}{\varrho}

\renewcommand{\SPHERE}{S}
\renewcommand{\BALL}{B}
\newcommand{\SPHEREzero}{\SPHERE^{(0)}}

\newcommand{\SubSet}{X}

\newcommand{\objective}{f}
\newcommand{\Objective}{\mathcal{F}}

\renewcommand{\Dual}{\mathbb{Y}}

\renewcommand{\TripleNorm}{\DoubleNorm}

\renewcommand{\DUAL}{\RR^{n}}
\renewcommand{\PRIMAL}{\RR^{n}}
\renewcommand{\Vset}{\ic{1,{n}}}

\renewcommand{\Capra}{Capra}

\graphicspath{{Figures/}}
\def\keywordsname{Keywords}
\def\mykeywords{\par\addvspace{17pt}\small\rmfamily
  \trivlist\if!\keywordsname!\item[]\else\item[\hskip\labelsep
    {\bfseries\keywordsname}]\fi}

\newenvironment{keywords}{\begin{mykeywords}}{\end{mykeywords}}

\title{What are \Capra-Convex Sets?}

\author{Jean-Philippe~Chancelier\thanks{CERMICS, CNRS, ENPC, Institut Polytechnique de Paris, Marne-la-Vall\'ee, France} 
  \and
  Michel~De~Lara\footnotemark[1]
  \and
  Adrien~Le~Franc\footnotemark[1]
  \and
  Seta Rakotomandimby\footnotemark[1]
}

\begin{document}
\maketitle


\begin{abstract}
  This paper focuses on a specific form of abstract convexity known as
  \Capra-convexity, where a constant along primal rays (Capra) coupling replaces
  the scalar product used in standard convex analysis to define generalized
  Fenchel conjugacies. A key motivating result is that the $\ell_0$
  pseudonorm --- which counts the number of nonzero components in a vector --- is equal
  to its \Capra-biconjugate. This implies that $\ell_0$ is a \Capra-convex
  function, highlighting potential applications in statistics and machine
  learning, particularly for enforcing sparsity in models. Building on prior
  work characterizing the \Capra-subdifferential of $\ell_0$ and the role of source
  norms in defining the \Capra-coupling, the paper provides
  a characterization of \Capra-convex sets.
\end{abstract}

\begin{keywords}
  Generalized subdifferential; $\lzero$ pseudonorm; Sparsity; \Capra-coupling
\end{keywords}
\section{Introduction}
\label{sec:introduction}%

From a historical perspective, convexity was first studied as a geometrical
property of sets~\cite{Fenchel:1983}.
The formal contemporary definition of a convex function was seemingly introduced
by Jensen~\cite{Jensen:1906}, in the context of a growing interest for such
functions
at the dawn of the XIX$^\text{th}$ century.
The development of modern convex analysis, with the prominent role of convex
conjugacies, is mostly due to Fenchel, Moreau and Rockafellar over the
XX$^\text{th}$ century, as exposed in the bibliographical notes
of~\cite{Hiriart-Urruty-Lemarechal:2004}.

Due to the success of convexity in solving optimization problems, there has been
several attempts to extend this theory to larger classes of sets, functions and
conjugacies.
For instance, starting from the geometrical roots of convexity, 
spherical convexity was introduced
for the study of sets contained in the Euclidean sphere:
the definition of a spherically-convex set follows the one of a classical convex set,
in the sense that the geodesic between two points of a spherical set
must remain inside the set~\cite{Ferreira-Iusem-Nemeth:2014}.
Conversely, other extensions of convexity originate from
abstract conjugacies,
which derive from the choice of a general set of functions
to play the role of affine minorants in the classical
definition of the Fenchel conjugate and biconjugate. 
Examples of abstract conjugacies are given as early as Moreau's reference lectures~\cite{Moreau:1966-1967}.
We refer to~\cite{Singer:1997} for a complete introduction to the topic.

In this paper, we focus on a specific type of abstract convexity, named \Capra-convexity,
where the so-called \Capra\ (constant along primal rays) coupling function
plays the role of the scalar product in usual convexity,
to generalize the set of affine minorants in usual Fenchel conjugacies.
A principal result of \Capra-convexity, which motivates our interest, 
is that the $\lzero$ pseudonorm --- which counts the nonzero components of a vector ---
is equal to its \Capra-biconjugate, and is therefore a \Capra-convex function~\cite{Chancelier-DeLara:2021_ECAPRA_JCA}.
This suggests potential applications to statistics and machine learning, 
where the $\lzero$ pseudonorm is extensively used to enforce sparsity
in statistical models.
From the perspective of convex analysis, further investigations
have allowed to better understand the role of a source norm
in the definition of the \Capra-coupling~\cite{Chancelier-DeLara:2022_OSM_JCA}
and to characterize the \Capra-subdifferential of $\lzero$~\cite{LeFranc-Chancelier-DeLara:2022}.
Moreover, algorithmic approaches based on a generalized
cutting plane method have been investigated recently in~\cite{Rakotomandimby:2025OL}.
In addition to these previous developments, and in the perspective
of minimizing $\lzero$ --- or any \Capra-convex function of interest --- 
over a constraint set which somehow \textit{preserves} \Capra-convexity, 
we now ask: 
\begin{center}
	what are \Capra-convex sets?
\end{center}

In order to address this question,
we first introduce background results on usual convexity 
and \Capra-convexity hereafter in~\S\ref{sec:introduction}.
Second, we propose a definition for \Capra-convex sets
in~\S\ref{sec:capra_convex_sets} and, as our main results,
we provide explicit characterizations of such sets.
Subsequently, 
we investigate on the relationship between
\Capra-convex sets and conical hulls, 
which allows a connection between closed spherically-convex and \Capra-convex sets.
Third, we propose a collection of examples in~\S\ref{sec:various_examples}
to illustrate our main results.
To ease the reading of the paper, 
most technical proofs are given in Appendix~\ref{sec:proofs}.

\subsection{Convex functions and sets}
\label{sec:background_sets}%

We introduce some basic notions 
and refer to~\cite{Rockafellar-Wets:1998, Bauschke-Combettes:2017}
for more advanced materials.

We denote by $\barRR = \RR \cup \na{+\infty, -\infty}$ the extended real line.
We consider the Euclidean space~$\PRIMAL$, equipped with the scalar product~$\proscal{\cdot}{\cdot}$.
For any function \( \fonctionuncertain \colon \PRIMAL \to \barRR \),
its \emph{epigraph} is the set \( \epigraph\fonctionuncertain= 
\defset{ \np{\uncertain,t}\in\PRIMAL\times\RR}%
{\fonctionuncertain\np{\uncertain} \leq t} \).
We recall that the function $\fonctionuncertain$ is convex
if and only if its epigraph is a convex set.
For any set $\SubSet \subseteq \PRIMAL$, 
$\Indicator{\SubSet} \colon \PRIMAL \to \barRR $ 
and 
$\sigma_\SubSet \colon \PRIMAL \to \barRR$ denote respectively
the \emph{indicator function} 
and the \emph{support function} of the
set~$\SubSet$, defined as
\begin{subequations}
  \begin{align}
    \Indicator{\SubSet}\np{\primal} &= 0 \text{ if } \primal \in \SubSet
                                      \eqsepv
                                      \Indicator{\SubSet}\np{\primal} = +\infty \text{ if } \primal \not\in \SubSet 
                                      \eqfinp 
                                      \label{eq:indicator_function}
    \\
    \sigma_{\SubSet}\np{\dual} &= \sup_{\primal \in \SubSet} \proscal{\primal}{\dual} \eqsepv \forall \dual \in \DUAL \eqfinp
    \label{eq:support_function}
  \end{align}
\end{subequations}
We denote by $\convexhull\np{\SubSet}$
the convex hull of the set $\SubSet$
(defined as the smallest convex subset of $\PRIMAL$ containing $\SubSet$),
and by $\closedconvexhull\np{\SubSet}$ the closed convex hull of $\SubSet$
(defined as the closure of $\convexhull\np{\SubSet}$).

We pay a special attention to cones, which are sets
$\Cone \subseteq \PRIMAL$ such that
\( \lambda\Cone\subset\Cone \) for all \( \lambda > 0 \).
Notice that, with this definition, a cone does not necessarily contain the origin,
and is not necessarily convex, as will be illustrated with \Capra-convex sets in~\S\ref{sec:capra_convex_sets}.
We say that a cone~$\Cone$ 
is \emph{pointed} if
\( \Cone \cap \np{-\Cone} \subseteq \na{0} \).
A cone can be constructed from
any subset $\SubSet \subseteq \PRIMAL$, by means of the \emph{conical hull}
of $\SubSet$, defined in~\cite[Definition 6.1]{Bauschke-Combettes:2017} as
\begin{subequations}
  \begin{align}
    \conicalhull(\SubSet)
    &=
      \defset{\lambda\primal}{\primal\in\SubSet\eqsepv \lambda> 0}
      \eqfinv
      \label{eq:conicalhull}
      \intertext{or by means of the \emph{positive hull}
      of $\SubSet$, defined in~\cite[Chapter 3, \S G]{Rockafellar-Wets:1998} as }
\positivehull(\SubSet)
   &=
     \defset{\lambda\primal}{\primal\in\SubSet\eqsepv \lambda \geq 0}
     =\conicalhull(\SubSet) \cup \na{0}
     \eqfinp
           \label{eq:positivehull}
  \end{align}
\end{subequations}
We stress that 
this definition of the conical hull omits the value
$\lambda = 0$, so that if $0 \notin \SubSet$
then $0 \notin \conicalhull(\SubSet)$.
By contrast, \( 0\in \positivehull(\SubSet) \). 
We also respectively denote by \(\closedcone\np{\SubSet}\) 
and \(\overline{\positivehull}\np{\SubSet}\) the topological closures of 
the conical hull and the positive hull of a set.

\subsection{Basic notions in \Capra-convexity}
\label{sec:backgounds_Capra}

We start by recalling the definition of the \Capra\ coupling.
In what follows, let $\DoubleNorm{\cdot}$ be 
a norm on~$\PRIMAL$, referred to as the \emph{source norm}.
We respectively denote 
\(\SPHERE= \defset{\primal \in \RR^d}{\norm{\primal} = 1} 
  \) 
  and 
  \(\BALL= \defset{\primal \in \RR^d}{\norm{\primal} \leq 1} 
  \) the unit sphere and unit ball associated with the source norm.

  \begin{definition}
    \label{de:normalization_mapping}
We define the set
\begin{equation}
  \SPHEREzero = \SPHERE \cup \na{0} \eqfinv
  \label{eq:sphere_zero}%
\end{equation}
and the \emph{radial projection}
\begin{equation}
  \RadialProjection \colon \PRIMAL \rightarrow \PRIMAL 
  \eqsepv
  \RadialProjection(\primal) =
  \frac{\normalsize \primal}{\normalsize \DoubleNorm{\primal}} \text{ if }  \primal \neq 0
  \eqsepv \RadialProjection(\primal) = 0  \text{ if }  \primal = 0 \eqfinp
  \label{eq:normalization_mapping}
\end{equation}    
\end{definition}
We will use the following easy-to-establish properties:
\begin{subequations}
	\begin{gather}
		\RadialProjection(\PRIMAL)=\SPHEREzero \eqfinv \label{eq:RadialProjection_PRIMAL} \\
		 \text{(zero-homogeneity)} \quad
		\RadialProjection(\lambda \primal) = \RadialProjection(\primal) \eqsepv
		\forall \lambda > 0 \eqsepv \primal \in \PRIMAL \eqfinv \label{eq:zero_homogeneity} \\
		\text{if } \Cone \subset\PRIMAL \text{ is a cone, then }
		\RadialProjection\np{\Cone} = \Cone \cap \SPHEREzero \eqfinv \label{eq:cone_inter_spherezero} \\
		\Converse{\RadialProjection}\np{\Primal}
		= \Converse{\RadialProjection}\np{\Primal \cap \SPHEREzero}
		=\conicalhull\np{\Primal \cap \SPHEREzero} \eqsepv \Primal \subset\PRIMAL \eqfinv \label{eq:RadialProjection_inverse} 
	\end{gather}
\end{subequations}

\begin{definition}[\cite{Chancelier-DeLara:2022_CAPRA_OPTIMIZATION}, Definition~4.1]
  \label{de:Capra}
  We define the
  \emph{\Capra~coupling}~$\CouplingCapra \colon \PRIMAL \times \DUAL \to \RR$ between
  $\PRIMAL$ and $\DUAL$, by
  \begin{equation}
    \forall \dual \in \DUAL \eqsepv 
    \CouplingCapra\np{\primal, \dual} 
    = \proscal{\RadialProjection(\primal)}{\dual} =
    \begin{cases}
      \frac{ \proscal{\primal}{\dual} }{ \DoubleNorm{\primal} }
      \eqsepv &\text{ if } \primal \neq 0 \eqfinv
      \\
      0 \eqsepv &\text{ if } \primal = 0 \eqfinp
    \end{cases}
    \label{eq:coupling_CAPRA}
  \end{equation}
\end{definition}

A coupling function such as the \Capra~coupling $\CouplingCapra$ given in Definition~\ref{de:Capra}
gives rise to generalized Fenchel-Moreau conjugacies~\cite{Singer:1997, Martinez-Legaz:2005},
that we briefly recall.
Let us consider a function $\fonctionprimal \colon \PRIMAL \to \barRR$.
The $\CouplingCapra$-Fenchel-Moreau conjugate of $\fonctionprimal$
is the function $\SFM{ \fonctionprimal }{\CouplingCapra} \colon \DUAL \to \barRR$
defined by
\begin{subequations}
  \label{eq:Capra_conjugacies}%
  \begin{equation}
    \SFM{ \fonctionprimal }{\CouplingCapra}(\dual)
    = 
    \sup_{\primal \in \PRIMAL} \Bp{ \CouplingCapra\np{\primal,\dual} 
      -\fonctionprimal\np{\primal}  } 
    \eqsepv \forall \dual \in \DUAL
    \eqfinv
    \label{eq:Fenchel-Moreau_conjugate}
  \end{equation}
  the $\CouplingCapra$-Fenchel-Moreau biconjugate of $\fonctionprimal$
  is the function $\SFMbi{ \fonctionprimal }{\CouplingCapra} \colon \PRIMAL \to \barRR$
  defined by
  \begin{equation}
    \SFMbi{ \fonctionprimal }{\CouplingCapra}(\primal)
    = 
    \sup_{\dual \in \DUAL} \Bp{ \CouplingCapra\np{\primal,\dual} 
      -\SFM{ \fonctionprimal }{\CouplingCapra}(\dual)  } 
    \eqsepv \forall \primal \in \PRIMAL
    \eqfinv
    \label{eq:Fenchel-Moreau_biconjugate}
  \end{equation}
  and we have the inequality
  \begin{equation}
    \SFMbi{ \fonctionprimal }{\CouplingCapra}(\primal)
    \leq \fonctionprimal(\primal)
    \eqsepv \forall \primal \in \PRIMAL
    \eqfinp
    \label{eq:biconjugate_inequality}
  \end{equation}
\end{subequations}

Observe that, if we replace the \Capra\ coupling $\CouplingCapra$ with the
scalar product $\proscal{\cdot}{\cdot}$ in \eqref{eq:Capra_conjugacies}, 
we retrieve the well-known notions of Fenchel conjugate $\SFM{ \fonctionprimal }{\star}$
and biconjugate $\SFMbi{ \fonctionprimal }{\star}$ in
standard convex analysis.  We refer to
\cite{Chancelier-DeLara:2022_CAPRA_OPTIMIZATION} for a more complete
introduction to \Capra\ conjugacies.

\begin{definition}
  \label{de:capra_convex_functions}
  We say that
  the function $\fonctionprimal \colon \PRIMAL \to \barRR$ is \Capra-convex
  \IFF we have an equality in~\eqref{eq:biconjugate_inequality}, that is,
  if $\SFMbi{ \fonctionprimal }{\CouplingCapra} = \fonctionprimal$.
\end{definition}

We define the \emph{support} of a vector $\primal=\sequence{\primal_j}{ j \in\Vset } \in \PRIMAL$ by
$\SupportMapping(\primal) = \bset{ j \in\Vset }{\primal_j \not= 0 }$,
 where $\Vset = \na{1, \ldots, n}$.
The \emph{\lzeropseudonorm} is the function
\( \lzero \colon \PRIMAL \to \Vset \)
defined~by 
\begin{equation}
  \lzero\np{\primal} =
  \bcardinal{ \SupportMapping(\primal) }
  \eqsepv \forall \primal \in \PRIMAL
  \eqfinv
  \label{eq:pseudo_norm_l0}  
\end{equation}
where $\cardinal{\IndexSubset}$ denotes the cardinality of 
a subset \( \IndexSubset \subseteq \Vset \).
We recall a central result which motivates our interest for \Capra-convexity: 
for a certain class of source norms, which encompasses e.g.
the $\ell_p$ norms for $1 < p < +\infty$,
the $\lzero$ pseudonorm is a \Capra-convex functions, in the sense of Definition~\ref{de:capra_convex_functions}.
We refer to~\cite{Chancelier-DeLara:2022_SVVA} for more details on the choice of a suitable source norm
to enforce the \Capra-convexity of $\lzero$.

\section{\Capra-convex sets}
\label{sec:capra_convex_sets}

In~\S\ref{sec:definition_and_main_results}, we define
\Capra-convex sets and present our main results regarding the characterization
of such sets.
Then, in \S\ref{sec:capra_convex_conical_hull}, 
we present sufficient conditions for the conical hull of a set to be \Capra-convex,
and we discuss of the links between \Capra-convex and spherically-convex sets.
To ease the reading of this section, the proofs of the main results 
are relegated to Appendix~\ref{sec:proofs}.

\subsection{Definition and characterization of \Capra-convex sets}
\label{sec:definition_and_main_results}

We start with the definition of \Capra-convex sets.
Subsequently, we outline our main results regarding the characterization of such sets.

\subsubsubsection{Definition of \Capra-convex sets using indicator functions}

Closed convex sets play a central role in
convex analysis: for instance,
proper functions with closed convex epigraphs
are stable by the Fenchel-Moreau biconjugate.
In particular, for a closed convex set $\SubSet \subseteq \PRIMAL$,
its indicator function $\Indicator\SubSet$ is convex \lsc, so that \(
   \Indicator\SubSet = \SFMbi{ \Indicator\SubSet }{\star}\), 
even for the --- unique --- nonproper case where 
$\SubSet = \emptyset$, as $\Indicator\emptyset = \SFMbi{ \Indicator\emptyset }{\star} = +\infty$.
By analogy with closed convex sets in classical convex analysis, we provide the following definition
of \Capra-convex sets.

\begin{definition}
  \label{de:capra_convex_sets}
  Let $\DoubleNorm{\cdot}$ be a source norm, and
  $\CouplingCapra$ be the corresponding \Capra-coupling
  as in Definition~\ref{de:Capra}.
  We say that the set $\SubSet \subseteq \PRIMAL$
  is \Capra-convex if the indicator function $\Indicator\SubSet$
  is a \Capra-convex function.
\end{definition}

With this definition, we deduce the following immediate property.

\begin{proposition}
  \label{pr:cone}%
  If a set $\SubSet \subseteq \PRIMAL$ is \Capra-convex, then $\SubSet$ is a cone.
\end{proposition}
\begin{proof}
	The coupling $\CouplingCapra$ in Definition~\ref{de:Capra}
	is one-sided linear (see~\cite[Definition~2.3]{Chancelier-DeLara:2021_ECAPRA_JCA}) and factorizes 
	with the radial projection $\RadialProjection \colon \PRIMAL \to \SPHEREzero$ in~\eqref{eq:normalization_mapping},
	so that
	$\Indicator\SubSet = \SFM{ \Indicator\SubSet }{\CouplingCapra\CouplingCapra'} =
	\SFM{ \Indicator\SubSet }{\CouplingCapra\star'} \circ \RadialProjection$,
	 see~\cite[Proposition 2.5]{Chancelier-DeLara:2021_ECAPRA_JCA}.
	 Thus, for any $\lambda > 0$, one has that 
         \[
           \Indicator{\lambda\SubSet}(\primal) =
\Indicator\SubSet(\primal/\lambda) = 
\SFM{ \Indicator\SubSet }{\CouplingCapra\star'}\bp{\RadialProjection(\primal/\lambda)} =
           \SFM{ \Indicator\SubSet }{\CouplingCapra\star'}\bp{\RadialProjection(\primal)} =
\Indicator\SubSet(\primal) 
\eqsepv \forall \primal\in\PRIMAL
\eqfinv
\]
where we have used the property that the radial projection~$\RadialProjection$ is zero-homogeneous.
We conclude that \( \lambda\SubSet=\SubSet \) for any $\lambda > 0$, 
which proves that $\SubSet$ is a cone.
\end{proof}

\subsubsubsection{Main results on the characterization of \Capra-convex sets}

We will see that a \Capra-convex set needs not be convex, and that, somehow more
surprisingly, it needs not be closed either.
Our principal result is the following generic characterization of \Capra-convex sets.

\begin{theorem}
  \label{th:set_inter_spherezero}%
  Let $\DoubleNorm{\cdot}$ be a source norm on $\PRIMAL$, 
  $\RadialProjection \colon \PRIMAL \to \SPHEREzero$ be the radial projection defined
  in~\eqref{eq:normalization_mapping}, and 
  $\CouplingCapra$ be the corresponding \Capra-coupling,
  as in Definition~\ref{de:Capra}.
  Let $\Cone \subseteq \PRIMAL$ be a set. 	We have that
  \begin{equation}
    \Cone \text{ is \Capra-convex }
    \iff
    \Cone \text{ is a cone and } 
    \RadialProjection\np{\Cone}
      = \closedconvexhull\bp{\RadialProjection\np{\Cone}} \cap \SPHEREzero
      \eqfinp
    \label{eq:set_inter_spherezero}%
  \end{equation}
\end{theorem}

To illustrate Theorem~\ref{th:set_inter_spherezero},
we provide an example of \Capra-convex set in Figure~\ref{fig:illustration_theorem_set_inter_spherezero}
which satisfies the characterization~\eqref{eq:set_inter_spherezero}.
Moreover,
to give a more geometrical interpretation of~\eqref{eq:set_inter_spherezero},
we recall~\eqref{eq:cone_inter_spherezero}:
if the set $\Cone$ is a cone, 
then $\RadialProjection\np{\Cone} = \Cone\cap\SPHEREzero$.
It is also insightful to observe that
the equality between sets in~\eqref{eq:set_inter_spherezero}
is equivalent to 
\(\Cone = \Converse{\RadialProjection}
\Bp{\closedconvexhull\bp{\RadialProjection\np{\Cone}}}\),
so that the set \(\Converse{\RadialProjection} 
\Bp{\closedconvexhull\bp{\RadialProjection\np{\Cone}}}\)
can be interpreted as the Capra-convex hull of $\Cone$.
Nevertheless, the characterization of~\eqref{eq:set_inter_spherezero}
remains quite formal, and we now
provide more practical conditions 
to identify \Capra-convex sets.

\begin{corollary}
  \label{co:implication}%
  Under the hypotheses of Theorem~\ref{th:set_inter_spherezero},
  we have that
  \begin{equation}
    \Cone \text{ is \Capra-convex }
    \implies
    \begin{cases}
      \Cone \text{ is a cone,} \\
      \Cone \cup \na{0} \text{ is closed} \eqfinv \\
      \Cone \cap \na{0} =
      \closedconvexhull\bp{\RadialProjection\np{\Cone}} \cap \na{0} \eqfinp 
    \end{cases}
    \label{eq:capra_convex_sets}%
  \end{equation}
\end{corollary}

Referring again to Figure~\ref{fig:illustration_theorem_set_inter_spherezero},
we observe that, for the example of set $\Cone$ in Figure~\ref{fig:K3_intersect_l_infty}, 
we have indeed that 
$\Cone \cap \na{0} =
\closedconvexhull\bp{\RadialProjection\np{\Cone}} \cap \na{0} = \emptyset$.
Thus, in practice, checking whether~\eqref{eq:capra_convex_sets} holds
often boils down to checking whether $0 \in \closedconvexhull\bp{\RadialProjection\np{\Cone}}$ holds,
which can be a convenient alternative to~\eqref{eq:set_inter_spherezero} 
--- we refer to the examples of~\S\ref{subsec:lp_examples}.

One may naturally wonder whether the reverse implication
holds in~\eqref{eq:capra_convex_sets}. Interestingly, the answer to that question
depends on the geometry of the unit ball induced by the source norm.
Indeed, we will show a sufficient condition to have the equivalence in~\eqref{eq:capra_convex_sets}
which depends on the following property of balls.
We recall that the unit ball~$\BALL$ 
of a norm~$\DoubleNorm{\cdot}$ is rotund when
the corresponding sphere~$\SPHERE$ coincides with
the extreme points of~$\BALL$.
For instance, this is the case of the~$\ell_p$ norms for values of~$p$
satisfying $1 < p < \infty$.

\begin{corollary}[rotund norm balls]
  \label{co:rotund_balls}%
  Under the hypotheses of Theorem~\ref{th:set_inter_spherezero}, suppose
  moreover that the unit ball of the source norm $\DoubleNorm{\cdot}$ is rotund.
  Then, the set~\(\Cone\) is \Capra-convex if and only if
  the three conditions in the right-hand side of~\eqref{eq:capra_convex_sets} are satisfied.
\end{corollary}

When the unit ball of the source norm $\DoubleNorm{\cdot}$ is not rotund,
the reverse implication in~\eqref{eq:capra_convex_sets} might fail,
as we illustrate with an example in~\S\ref{subsec:lp_examples}.
Lastly,
in addition to Corollary~\ref{co:rotund_balls},
and regardless of the source norm $\DoubleNorm{\cdot}$, 
we identify the following notable cases
of \Capra-convex sets.

\begin{corollary}[closed convex cones]
  \label{co:closed_convex_cones}%
  Let $\Cone \subseteq \PRIMAL$ be a closed convex cone.
  We have that
  \begin{itemize}
  	\item[$(i)$] the set~\(\Cone\) is \Capra-convex,
  	\item[$(ii)$] if, moreover, the cone~$\Cone$ is pointed, then the set $\Cone\setminus\na{0}$
  	is \Capra-convex. 
  \end{itemize}
\end{corollary}

\begin{figure}
	\begin{subfigure}[b]{0.5\linewidth}
		\centering
		\begin{tikzpicture}[scale=0.8]
		\filldraw[blue!30] (0,0) -- (2.5*1/2, {2.5*sqrt(3)/2}) -- (2.5*1/2, {-2.5*sqrt(3)/2}) -- cycle;
		\draw[gray, dashed] (-1,-1)--(-1,1) ;
		\draw[gray, dashed] (-1,1)--(1,1) ; 
		\draw[gray, dashed] (1,1)--(1,-1) ; 
		\draw[gray, dashed] (1,-1)--(-1,-1) ; 
		\draw[->] (-3,0)--(3,0) node[below right]{{$\primal_1$}};
		\draw[->] (0,-3)--(0,3) node[above left]{{$\primal_2$}};
		\draw[blue, thick] (0,0)--(2.5*1/2, {2.5*sqrt(3)/2});
		\draw[blue, thick] (0,0)--(2.5*1/2, {-2.5*sqrt(3)/2});
		\draw[black, thick] (0,0) circle (2pt);
		\filldraw[white, thick] (0,0) circle (1.5pt);
		\draw[red, thick] (1,-1) -- (1, 1);
		\draw[red, thick] ({1/sqrt(3)}, 1) -- (1, 1);
		\draw[red, thick] ({1/sqrt(3)}, -1) -- (1, -1);
		\end{tikzpicture} 
		\caption{\(\Cone\) in blue and \(\RadialProjection\np{\Cone}\) in red}
		\label{fig:K3_intersect_l_infty}
	\end{subfigure}
	\hfill
	\begin{subfigure}[b]{0.5\linewidth}
		\centering
		\begin{tikzpicture}[scale=0.8]
		\draw[gray, dashed] (-1,-1)--(-1,1) ;
		\draw[gray, dashed] (-1,1)--(1,1) ; 
		\draw[gray, dashed] (1,1)--(1,-1) ; 
		\draw[gray, dashed] (1,-1)--(-1,-1) ;
		\filldraw[cyan!30] ({1/sqrt(3)}, 1) -- ({1/sqrt(3)}, -1) 
		-- (1,-1) -- (1, 1) -- cycle ;
		\draw[cyan, thick] ({1/sqrt(3)}, 1) -- ({1/sqrt(3)}, -1) 
		-- (1,-1) -- (1, 1) -- cycle ;
		\draw[red, thick] (1,-1) -- (1, 1);
		\draw[red, thick] ({1/sqrt(3)}, 1) -- (1, 1);
		\draw[red, thick] ({1/sqrt(3)}, -1) -- (1, -1);
		\draw[->] (-3,0)--(3,0) node[below right]{{$\primal_1$}};
		\draw[->] (0,-3)--(0,3) node[above left]{{$\primal_2$}};
		%
		\end{tikzpicture}
		\caption{$\closedconvexhull\bp{\RadialProjection\np{\Cone}}$ in cyan
			and \(\closedconvexhull\bp{\RadialProjection\np{\Cone}} \cap \SPHEREzero \) in red.
		} 
		\label{fig:co_n_K3_intersect_l_infty}  
	\end{subfigure}
	\caption{Illustration of Theorem~\ref{th:set_inter_spherezero}. Example of a set~\(\Cone\) for which \(\RadialProjection\np{\Cone}\)
		(in red in Figure~\ref{fig:K3_intersect_l_infty})
		is equal to \(\closedconvexhull\bp{\RadialProjection\np{\Cone}} \cap \SPHEREzero \)
		(in red in Figure~\ref{fig:co_n_K3_intersect_l_infty})
		when the source norm is~\(\Norm{\cdot}_\infty\)}
	\label{fig:illustration_theorem_set_inter_spherezero}
\end{figure}

\subsection{\Capra-convex conical hulls and relationship with spherical convexity}
\label{sec:capra_convex_conical_hull}%

First, we discuss in~\S\ref{Closed_convex_factorization_of_Capra-convex_problems}
the role of \Capra-convex conical hulls in optimization problems. 
Second,
in~\S\ref{Sufficient_conditions_for_a_set_to_have_a_Capra-convex_conical_hull} 
we give sufficient conditions for a set to have a \Capra-convex conical hull. 
Finally, we end in~\S\ref{relationship_with_spherical_convexity}
with a comparison between \Capra-convex sets and spherically-convex sets.

\subsubsection{Closed convex factorization of \Capra-convex problems}
\label{Closed_convex_factorization_of_Capra-convex_problems}

\begin{subequations}
\Capra-convex constrained minimization problems enjoy a closed convex factorization property.
More specifically, let us consider the optimization problem
\begin{equation}
  \inf_{\primal \in \SubSet} \objective(\primal)
  \eqfinv
\label{eq:optimization}
\end{equation}
defined after an objective function
$\objective \colon \PRIMAL \to \RR$ and a constraint set $\SubSet \subseteq \PRIMAL$.
If both the objective function
$\objective$
and the indicator function 
$\Indicator\SubSet$ 
are \Capra-convex, 
then we recall, according to
\cite[Proposition 2.6]{Chancelier-DeLara:2021_ECAPRA_JCA}, that the function 
$\objective + \Indicator\SubSet$
factorizes as $\objective + \Indicator\SubSet 
= \Objective \circ \RadialProjection$, 
where $\Objective \colon \PRIMAL \to \barRR$
is a proper \lsc\ convex function, so that Equation~\eqref{eq:optimization} becomes
\begin{equation}
\inf_{\primal \in \SubSet} \objective(\primal)
=
\inf_{\primal \in \PRIMAL} \Objective \circ \RadialProjection\np{\primal}
=
\inf_{\sphere \in \SPHEREzero} \Objective \np{\sphere}
=\inf\bp{\Objective
  \np{0},\inf_{\sphere \in \SPHERE} \Objective \np{\sphere} }
\eqfinp
\label{eq:closed_convex_over_sphere}
\end{equation}
In such a case, solving the optimization Problem~\eqref{eq:optimization}
amounts to minimizing the closed convex~\cite[p.~15]{Rockafellar:1974} function 
$\Objective $ over the unit sphere~$\SPHERE$. 
\end{subequations}

However, minimization problems involving \Capra-convex objective functions are usually \emph{not} defined on a \Capra-convex set of constraints.
For instance, among minimal cardinality problems, i.e. instances with $\objective = \lzero$ in Problem~\eqref{eq:optimization},
the fundamental sparse problem in compressive sensing has a constraint set
of type $\SubSet = \defset{\primal \in \PRIMAL}{A\primal = b}$ (see e.g.~\cite{Foucart-Rauhut:2013}).
Yet, when $b\neq0$, the resulting affine space is not a cone, so that $\SubSet$ 
cannot be a \Capra-convex set according to
Proposition~\ref{pr:cone}.

Nonetheless, \Capra-convex functions~\(\fonctionprimal \colon \PRIMAL \to \barRR\) are \(0\)-homogeneous, 
that \(\fonctionprimal\np{\lambda \primal} = \fonctionprimal\np{\primal}\), for any \(\primal \in \PRIMAL\) and any real number~\(\lambda >0\). 
Thus, for such functions, the problem~\eqref{eq:optimization} can be rewritten as
\begin{equation}
  \label{eq:capra_convexification} 
  \inf_{\bisprimal \in \conicalhull\np{\SubSet}} \objective(\bisprimal) \eqfinv
\end{equation}
using the change of variable \(\primal = \lambda \bisprimal, \bisprimal \in \PRIMAL,\lambda >0\). 
Therefore, even if the set of constraints~\(\SubSet\) is
 not \Capra-convex, its conical hull~\(\conicalhull\np{\SubSet}\) 
 might be \Capra-convex, 
 under some assumptions that we now investigate on.

\subsubsection{Sufficient conditions for a set to have a \Capra-convex conical
  hull}
\label{Sufficient_conditions_for_a_set_to_have_a_Capra-convex_conical_hull}

We now provide conditions for a set 
$\SubSet \subseteq \PRIMAL$ to have a \Capra-convex conical hull.
We restrict our interest to cases where $0 \notin \SubSet$,
which correspond to nontrivial instances
of minimal cardinality problems, where $\objective = \lzero$ in Problem~\eqref{eq:optimization}.

\begin{proposition}
  \label{pr:coneX}
  Let $\SubSet \subseteq \PRIMAL$
  be a compact set such that $0 \notin \convexhull(\SubSet)$.
  Let $\DoubleNorm{\cdot}$ be a source norm on $\PRIMAL$
  and $\CouplingCapra$ be the corresponding
  Capra-coupling as in Definition~\ref{de:Capra}.
  
  If one of the following conditions is satisfied:
  \begin{itemize}
    \item[\((i)\)] the unit ball of the source norm~$\DoubleNorm{\cdot}$  is rotund;

    \item[\((ii)\)] the set~\(\SubSet\) is convex;
  \end{itemize} 

  then $\conicalhull(\SubSet)$ is \Capra-convex.
\end{proposition}

We introduce two examples to emphasize that,
without the two key assumptions of Proposition~\ref{pr:coneX},
the conical hull of a set $\SubSet \subseteq \PRIMAL$ might fail to be \Capra-convex.

We start with an example where $0 \in \convexhull(\SubSet)$.

\begin{example}
  Let the set $\SubSet \subset \RR^2$ be the
  ball of center $(1, 0)$ and radius $1$.
  This set is closed, convex and bounded
  with $0 \in \SubSet$.
  Its conical hull is 
  $\conicalhull\np{\SubSet} = \defset{\primal \in \RR^2}{\primal_1 > 0} \cup \na{0}$,
  which is not \Capra-convex from Corollary~\ref{co:rotund_balls},
  as $\conicalhull\np{\SubSet} \cup \na{0}$
  is not closed (consider for instance the sequence \(\na{\np{1/k,1}}_{k\geq 1} \subset \conicalhull\np{\SubSet} \cup \na{0}\)).
\end{example}

Similarly, we give an example to illustrate that 
the conical hull $\conicalhull\np{\SubSet}$ might fail to be \Capra-convex, if the 
set~\(\SubSet\) is not compact.
In the following example, we get back to the 
fundamental sparse minimization problem from compressive sensing~\cite{Foucart-Rauhut:2013}.

\begin{example}[Figure~\ref{fig:affine_space_without_bounds}]
  Let the set~$\SubSet$ in Problem~\eqref{eq:optimization} be the affine space $H$ of solutions
  of the linear system $Ax = b$,   with a nonzero matrix $A \in \RR^{m\times n}$ and a nonzero vector $b \in \RR^m$.
  If $\text{ker}A \neq \na{0}$, then $\conicalhull(\SubSet)$ is not \Capra-convex.
  \label{ex:affine_space}%
\end{example}

\begin{proof}
  Let $\primal \in \text{ker}A$ be nonzero. Let $\bar{\primal} \in \SubSet$.
  Let $\sequence{\varepsilon_k}{k \in \NN}$ be a positive sequence converging to~$0$.
  For $k \in \NN$, we define $\primal_k = \primal + \varepsilon_k \bar{\primal}$ and
  $\primal_k' = -\primal + \varepsilon_k \bar{\primal}$.
  As $A(\frac{\primal_k}{\varepsilon_k}) = A(\frac{\primal}{\varepsilon_k}) + A(\bar{\primal})=
0+b = b=A(\frac{\primal_k'}{\varepsilon_k})$,
  we get that $\na{\primal_k, \primal_k'} \subset \conicalhull(\SubSet)$.
  It follows that 
  $\na{\RadialProjection(\primal_k), \RadialProjection(\primal_k')} \subset \conicalhull(\SubSet) \cap \SPHEREzero$
  as $\RadialProjection(\PRIMAL)=\SPHEREzero$.
  Now, we observe that
  \begin{align*}
    \frac{1}{2} \RadialProjection(\primal_k) + \frac{1}{2} \RadialProjection(\primal_k')
    &= 
    \frac{1}{2} \Bp{
      \frac{\primal + \varepsilon_k \bar{\primal}}{\TripleNorm{\primal + \varepsilon_k \bar{\primal}}}
      - \frac{\primal - \varepsilon_k \bar{\primal}}{\TripleNorm{\primal - \varepsilon_k \bar{\primal}}}
    }
      \tag{by~\eqref{eq:normalization_mapping}, as \( \primal_k=\primal + \varepsilon_k \bar{\primal} \neq 0 \) since \( A(\frac{\primal_k}{\varepsilon_k}) = b\neq 0 \),
      and the same for \( \primal - \varepsilon_k \bar{\primal} \)} 
      \\
    &\rightarrow_{k\to +\infty}
      \frac{1}{2} \Bp{ \frac{\primal}{\TripleNorm{\primal}} - \frac{\primal}{\TripleNorm{\primal}} } =0
    \eqfinp    
  \end{align*}
  We deduce that 
  $0 \in \closedconvexhull\bp{\conicalhull(\SubSet) \cap \SPHEREzero}$.
  However,
  $0 \notin \conicalhull(\SubSet)$, as
  $A(0) = 0 \neq b$ hence $0 \notin \SubSet$.
  Thus, using~\eqref{eq:cone_inter_spherezero}, we obtain that the third
  condition in the right-hand-side of~\eqref{eq:capra_convex_sets}
  is not satisfied by $\conicalhull(\SubSet)$.
  We conclude that the set $\conicalhull(\SubSet)$
  is not \Capra-convex, from Corollary~\eqref{co:implication}.
\end{proof}

Lastly, regarding Example~\ref{ex:affine_space},
we observe that if we work with a bounded
subset $\SubSet \subset H$ of the affine space $H$,
we retrieve the property that $\conicalhull(\SubSet)$ is \Capra-convex,
from Proposition~\ref{pr:coneX}. This case is illustrated in Figure~\ref{fig:affine_space_with_bounds}.

\begin{figure}[H]
    \begin{subfigure}[b]{0.5\linewidth}
      \centering
      \begin{tikzpicture}[scale=1.0]
        \begin{axis}[ 
          ticks=none,
          axis lines = middle,
          axis line style={->},
          ymin=-2, ymax=5,
          xmin=-1, xmax=6,
          xlabel={$\primal_1$},
          ylabel={$\primal_2$},
          x label style={at={(axis description cs:1.05,0.25)},anchor=center},
           y label style={at={(axis description cs:0.2,1.0)},anchor=center},
          axis equal image
          ]
          
          \draw[gray, dashed] (axis cs:0,0) circle [radius=1];
          

          \addplot[name path=affine_H, black, thick, domain=-1:6] {4-x};
          \addplot[black] coordinates {(1.5,3)} node[above]{$H$} ;

          \addplot[name path=ray-, blue,dashed, domain=-1:6] {-x};
          \addplot[name path=ray+, blue!30, domain=-1:6] {0*x +5};
          \addplot[name path=rayB, blue!30, domain=2:6] {0*x -2};
          \addplot[color=blue!30, opacity=0.5]fill between[of=ray+ and ray-, soft clip={domain=-1:2}];
          \addplot[color=blue!30, opacity=0.5]fill between[of=ray+ and rayB, soft clip={domain=2:6}];
          \addplot[blue] coordinates {(4,2.5)} node[below]{$ \conicalhull \np{H} $} ;

          \addplot[black, thick, mark=*, only marks, mark size=2pt] coordinates {(0,0)} ;
          \addplot[white, thick, mark=*, only marks, mark size=1.5pt] coordinates {(0,0)} ;
              
        \end{axis}
      \end{tikzpicture}
      \caption{Affine space $H$ in black and $ \conicalhull \np{H} $ in blue}
      \label{fig:affine_space_without_bounds}
    \end{subfigure}
    \hfill
    \begin{subfigure}[b]{0.5\linewidth}
    	\centering
    	\begin{tikzpicture}[scale=1.0]
    	\begin{axis}[ 
    	ticks=none,
    	axis lines = middle,
    	axis line style={->},
    	ymin=-2, ymax=5,
    	xmin=-1, xmax=6,
    	xlabel={$\primal_1$},
    	ylabel={$\primal_2$},
    	x label style={at={(axis description cs:1.05,0.25)},anchor=center},
    	y label style={at={(axis description cs:0.2,1.0)},anchor=center},
    	axis equal image
    	]
    	
    	\draw[gray, dashed] (axis cs:0,0) circle [radius=1];
    	
    	
    	\addplot[name path=affine_H, black, thick, domain=-1:3] {4-x};
    	\addplot[name path=affine_H, black, thick, domain=5:6] {4-x};
    	\addplot[black] coordinates {(1.5,3)} node[above]{$H$} ;
    	
    	\addplot[name path=upp_bound, black, dashed, domain=-1:6] {1};
    	\addplot[name path=low_bound, black, dashed, domain=-1:6] {-1};

    	\addplot[name path=ray+, blue, domain=0:6] {0.33*x};
    	\addplot[name path=ray-, blue, domain=0:6] {-0.2*x};
    	\addplot[color=blue!30, opacity=0.5]fill between[of=ray+ and ray-, soft clip={domain=0:6}];
    	
    	\addplot[name path=affine_H, red, thick, domain=3:5] {4-x};
    	\addplot[red] coordinates {(4,-1)} node[below]{$\SubSet$} ;
    	\addplot[blue] coordinates {(4,2.5)} node[below]{$ \conicalhull \np{\SubSet} $} ;
    	
    	\addplot[black, thick, mark=*, only marks, mark size=2pt] coordinates {(0,0)} ;
    	\addplot[white, thick, mark=*, only marks, mark size=1.5pt] coordinates {(0,0)} ;
    	
    	\end{axis}
    	\end{tikzpicture}
    	\caption{Bounded subset $\SubSet$ in red and $ \conicalhull \np{\SubSet} $ in blue}
    	\label{fig:affine_space_with_bounds}
    \end{subfigure}
    \caption{Illustration of Example~\ref{ex:affine_space}, 
	where $\conicalhull\np{H}$ in Fig.~\ref{fig:affine_space_without_bounds} (left) is not \Capra-convex,
	but $\conicalhull\np{\SubSet}$ in Fig.~\ref{fig:affine_space_with_bounds} (right) is \Capra-convex}
    \label{fig:affine_space}
  \end{figure}

\subsubsection{Relationship with spherical convexity}
\label{relationship_with_spherical_convexity}

On the one hand, \emph{spherically-convex sets} are defined using the unit sphere~\(\SPHERE\) and the associated radial projection~\(\RadialProjection \colon \PRIMAL \to \SPHEREzero\)
\cite{Ferreira-Iusem-Nemeth:2014,Guo-Peng:2021}.
Indeed, following~\cite[Definition 2.5]{Guo-Peng:2021}\footnote{We use the definition of spherically-convex from~\cite{Guo-Peng:2021} where more general spheres are considered than in~\cite{Ferreira-Iusem-Nemeth:2014} where the focus is on the Euclidean sphere.
The equivalence of the definitions for the Euclidean sphere is proved by the characterizations~\cite[Proposition~2]{Ferreira-Iusem-Nemeth:2014} and~\cite[Proposition~2.7~(iv)]{Guo-Peng:2021}.
}
spherically-convex sets are subsets \(\SubSet \subset \SPHERE\) of the unit sphere~$\SPHERE$
that are defined by 
\begin{equation}
\SubSet \text{ is spherically-convex if }
\RadialProjection\bp{
  \lambda \primal + \np{1-\lambda} \primalbis
} \in \SubSet
\eqsepv 
\forall \primal, \primalbis \in \SubSet
\eqsepv \forall \lambda \in \ClosedIntervalClosed{0}{1}
\eqfinp
\label{eq:spherical_convexity}
\end{equation}

On the other hand, the unit sphere~\(\SPHERE\) and the associated radial projection~\(\RadialProjection \colon \PRIMAL \to \SPHEREzero\) also appear in the 
characterization~\eqref{eq:set_inter_spherezero} of \Capra-convex sets in Theorem~\ref{th:set_inter_spherezero}.
Hence, it begs the question: \emph{is there a link between \Capra-convex sets and spherically-convex sets?}

First, \Capra-convex sets are necessarily cones in~$\PRIMAL$ while spherically-convex sets are subsets of the unit sphere~$\SPHERE$. Thus, to make a relevant comparison, we will consider \Capra-convex sets and conical hulls in~$\PRIMAL$ of spherically-convex sets. We answer the following questions.
\begin{enumerate}
\item
  Is the conical hull of a spherically-convex set a \Capra-convex set? Not necessarily. 
  On the one hand, a spherically-convex set~\(\SubSet\) is not necessarily closed, according to the characterization~\cite[Proposition  2.7 (iv)]{Guo-Peng:2021}, so \(\conicalhull\np{\SubSet} \cup \na{0}\) is not necessarily closed.
  On the other hand, a \Capra-convex set~\(\Cone\) is such that \(\Cone \cup \na{0}\) is closed, according to Corollary~\ref{co:rotund_balls}.
  \item
    Is the topological closure of the conical hull of a spherically-convex set a \Capra-convex set? Yes, see Proposition~\ref{pr:spherical_convex_implies_capra_convex}.

  \item
    Is the intersection of a \Capra-convex set with the unit sphere 
    a spherically-convex set? Not necessarily. See Figure~\ref{fig:K2} which gives a counterexample of a \Capra-convex set~\(\Cone\) 
    which is not the conical hull of some spherically-convex set. Indeed, the set \(\Cone\) is not convex, while conical hulls of spherically-convex sets are convex, according to~\cite[Proposition~2]{Ferreira-Iusem-Nemeth:2014}.
\end{enumerate}

\begin{proposition}
  If the set $\SubSet \subset \SPHERE$ is spherically-convex, then
   \(\closedcone\np{\SubSet}\) 
   is \Capra-convex.  
   Moreover, if \(\closedcone\np{\SubSet}\) is pointed 
   then \(\closedcone\np{\SubSet} \backslash \na{0} \) is \Capra-convex.
  \label{pr:spherical_convex_implies_capra_convex}
\end{proposition}
\begin{proof}
  Using the characterization~\cite[Proposition  2.7 (iv)]{Guo-Peng:2021},
 if \(\SubSet\) is spherically-convex then its positive hull~\(\positivehull\np{\SubSet} \subset \PRIMAL\) is convex and pointed. We therefore have that
  \( \overline{\positivehull}\np{\SubSet}\) is a closed convex cone
  and as \( \overline{\positivehull}\np{\SubSet} = \closedcone \np{\SubSet} \)
  that \(\overline{\positivehull}\np{\SubSet}\) is \Capra-convex  by Corollary~\ref{co:closed_convex_cones}, Item~\((i)\).
   Moreover if \( \closedcone\np{\SubSet}\) remains pointed we obtain that
  \(\closedcone\np{\SubSet} \backslash \na{0} \) is \Capra-convex  by Corollary~\ref{co:closed_convex_cones}, Item~\((i)\).
\end{proof}

\section{Examples of \Capra-convex sets}
\label{sec:various_examples}

In this section, we showcase examples of \Capra-convex sets.
We start with generic examples of \Capra-convex sets for the \(\ell_p\) source norms in~\S\ref{subsec:lp_examples},
and then discuss the properties of the
the sublevel sets and the epigraph of \Capra-convex functions in~\S\ref{subsec:Capra-convex_epigraphs_and_sublevel_sets_of_functions}.

\subsection{Examples with the $\ell_p$ source norms}
\label{subsec:lp_examples}%

We now consider the source norms $\DoubleNorm{\cdot} = \Norm{\cdot}_p$
with $p \in \nc{1,\infty}$.
Our interest in this type of norms
is based on a thorough analysis of the
\Capra-convexity of~$\lzero$ and of the expression
of its \Capra-subdifferential for 
all values of $p \in \nc{1, \infty}$,
as exposed in~\cite{LeFranc-Chancelier-DeLara:2022}.
In particular, we concentrate on the cases
$p = 2$, for which the unit ball $\BALL_2$ is rotund,
and on the case $p = \infty$, for which 
the unit ball $\BALL_\infty$ fails to be rotund.

To illustrate the characterizations obtained
in~\S\ref{sec:definition_and_main_results},
we provide examples of cones in $\RR^2$ and comment on their \Capra-convexity.
Bearing in mind potential applications to sparse optimization,
we are specifically interested in cones $\Cone \subseteq \PRIMAL$
such that $0\notin\Cone$.
Indeed, when minimizing the \lzeropseudonorm\ on a set or, equivalently,
  on the conical hull of this set, we are not interested in trivial cases 
  where the minimum is attained at the origin.
Thus, we concentrate on the following examples:
\begin{subequations}
  \label{eq:cone_examples}%
  \begin{align}
    \Cone_1 &=
              \conicalhull\bp{\na{(1, 0), (-1, 1), (-1, -1)}}
              \text{\qquad (Figure~\ref{fig:K1})}
              \eqfinv
    \\ 
    \Cone_2 &=
              \conicalhull\bp{\na{(-1, 0), (-1, 1), (-1, -1)}}
              \text{\qquad (Figure~\ref{fig:K2})}
              \eqfinv
    \\ 
    \Cone_3 &=
              \conicalhull\Bp{{\convexhull\bp{\na{(1/2, \sqrt{3}/2), (1/2, -\sqrt{3}/2)}}}}
              \text{\qquad (Figure~\ref{fig:K3})}
              \eqfinp
  \end{align}
\end{subequations}

\subsubsubsection{Examples with the rotund unit ball of the $\ell_2$ norm $\Norm{\cdot}_2$}

Let us consider the source norm $\DoubleNorm{\cdot} = \Norm{\cdot}_2$.
We recall that, with this choice of source norm, the unit ball
$\BALL$ is rotund, so that we can rely on
Corollary~\ref{co:rotund_balls}
to characterize \Capra-convex sets.

We discuss the case of the three cones 
$\na{\Cone_i}_{i\in\na{1,2,3}}$
introduced in~\eqref{eq:cone_examples} 
and of the closed convex hull
of their image by the radial projection
$\RadialProjection$ in~\eqref{eq:normalization_mapping}, 
under the source norm $\DoubleNorm{\cdot} = \Norm{\cdot}_2$.
These sets are illustrated in Figure~\ref{fig:cones}.

For all three cones $\na{\Cone_i}_{i\in\na{1,2,3}}$, we have that
$\Cone_i \cup \na{0}$ is closed
and $0 \notin \Cone_i$.
It follows that,
according to Corollary~\ref{co:rotund_balls},
we only have to check the condition
$\closedconvexhull\bp{\RadialProjection\np{\Cone_i}} \cap \na{0} = \emptyset$
to assert the \Capra-convexity of these cones.

\begin{itemize}
\item For the cone $\Cone_1$ in Figure~\ref{fig:K1},
  we have that 
  $\closedconvexhull\bp{\RadialProjection\np{\Cone_1}} \cap \na{0} = \na{0}$
  in Figure~\ref{fig:co_n_K1}.
  We conclude from Corollary~\ref{co:rotund_balls}
  that $\Cone_1$ is not \Capra-convex for the choice
  of source norm $\DoubleNorm{\cdot} = \Norm{\cdot}_2$.
  
\item For the cone $\Cone_2$ in Figure~\ref{fig:K2},
  we have that 
  $\closedconvexhull\bp{\RadialProjection\np{\Cone_2}} \cap \na{0} = \emptyset$
  in Figure~\ref{fig:co_n_K2}.
  We conclude from Corollary~\ref{co:rotund_balls}
  that $\Cone_2$ is \Capra-convex for the choice
  of source norm $\DoubleNorm{\cdot} = \Norm{\cdot}_2$.

\item For the cone $\Cone_3$ in Figure~\ref{fig:K3},
  we have that 
  $\closedconvexhull\bp{\RadialProjection\np{\Cone_3}} \cap \na{0} = \emptyset$
  in Figure~\ref{fig:co_n_K3}.
  We conclude from Corollary~\ref{co:rotund_balls}
  that $\Cone_3$ is \Capra-convex for the choice
  of source norm $\DoubleNorm{\cdot} = \Norm{\cdot}_2$.
\end{itemize}

\begin{figure}[htbp]
  \begin{subfigure}[b]{0.5\linewidth}
    \centering
    \begin{tikzpicture}[scale=0.8]
      \draw[gray, dashed] (0,0) circle (1.0); 
      \draw[->] (-3,0)--(3,0) node[below right]{{$\primal_1$}};
      \draw[->] (0,-3)--(0,3) node[above left]{{$\primal_2$}};
      \draw[blue, thick] (0,0)--({sqrt(8)},0);
      \draw[blue, thick] (0,0)--(-2,2);
      \draw[blue, thick] (0,0)--(-2,-2);
      \draw[black, thick] (0,0) circle (2pt);
      \filldraw[white, thick] (0,0) circle (1.5pt);
    \end{tikzpicture} 
    \caption{$K_1$ is not \Capra-convex for $\DoubleNorm{\cdot} = \Norm{\cdot}_2$}
    \label{fig:K1} 
  \end{subfigure}
  \hfill
  \begin{subfigure}[b]{0.5\linewidth}
    \centering
    \begin{tikzpicture}[scale=0.8]
      \draw[gray, dashed] (0,0) circle (1.0); 
      \filldraw[cyan!30] (1,0) -- ({-sqrt(2)/2},{sqrt(2)/2}) -- ({-sqrt(2)/2},{-sqrt(2)/2}) -- cycle ;
      \draw[cyan, thick] (1,0) -- ({-sqrt(2)/2},{sqrt(2)/2}) -- ({-sqrt(2)/2},{-sqrt(2)/2}) -- cycle ;
      \draw[->] (-3,0)--(3,0) node[below right]{{$\primal_1$}};
      \draw[->] (0,-3)--(0,3) node[above left]{{$\primal_2$}};
      \filldraw[red, thick] (0,0) circle (3pt);
    \end{tikzpicture}
    \caption{$0 \in \closedconvexhull\bp{\RadialProjection\np{\Cone_1}}$ 
      for $\DoubleNorm{\cdot} = \Norm{\cdot}_2$ } 
    \label{fig:co_n_K1} 
  \end{subfigure}
  
  \vskip\baselineskip
  
  \begin{subfigure}[b]{0.5\linewidth}
    \centering
    \begin{tikzpicture}[scale=0.8]
      \draw[gray, dashed] (0,0) circle (1.0); 
      \draw[->] (-3,0)--(3,0) node[below right]{{$\primal_1$}};
      \draw[->] (0,-3)--(0,3) node[above left]{{$\primal_2$}};
      \draw[blue, thick] (0,0)--({-sqrt(8)},0);
      \draw[blue, thick] (0,0)--(-2,2);
      \draw[blue, thick] (0,0)--(-2,-2);
      \draw[black, thick] (0,0) circle (2pt);
      \filldraw[white, thick] (0,0) circle (1.5pt);
    \end{tikzpicture} 
    \caption{$K_2$ is \Capra-convex for $\DoubleNorm{\cdot} = \Norm{\cdot}_2$}
    \label{fig:K2} 
  \end{subfigure}
  \hfill
  \begin{subfigure}[b]{0.5\linewidth}
    \centering
    \begin{tikzpicture}[scale=0.8]
      \draw[gray, dashed] (0,0) circle (1.0); 
      \filldraw[cyan!30] (-1,0) -- ({-sqrt(2)/2},{sqrt(2)/2}) -- ({-sqrt(2)/2},{-sqrt(2)/2}) -- cycle ;
      \draw[cyan, thick] (-1,0) -- ({-sqrt(2)/2},{sqrt(2)/2}) -- ({-sqrt(2)/2},{-sqrt(2)/2}) -- cycle ;
      \draw[->] (-3,0)--(3,0) node[below right]{{$\primal_1$}};
      \draw[->] (0,-3)--(0,3) node[above left]{{$\primal_2$}};
      \filldraw[mygreen, thick] (0,0) circle (3pt);
    \end{tikzpicture}
    \caption{$0 \notin \closedconvexhull\bp{\RadialProjection\np{\Cone_2}}$  
      for $\DoubleNorm{\cdot} = \Norm{\cdot}_2$ } 
    \label{fig:co_n_K2} 
  \end{subfigure}
  
  \vskip\baselineskip
  
  \begin{subfigure}[b]{0.5\linewidth}
    \centering
    \begin{tikzpicture}[scale=0.8]
      \filldraw[blue!30] (0,0) -- (2.5*1/2, {2.5*sqrt(3)/2}) -- (2.5*1/2, {-2.5*sqrt(3)/2}) -- cycle;
      \draw[gray, dashed] (0,0) circle (1.0); 
      \draw[->] (-3,0)--(3,0) node[below right]{{$\primal_1$}};
      \draw[->] (0,-3)--(0,3) node[above left]{{$\primal_2$}};
      \draw[blue, thick] (0,0)--(2.5*1/2, {2.5*sqrt(3)/2});
      \draw[blue, thick] (0,0)--(2.5*1/2, {-2.5*sqrt(3)/2});
      \draw[black, thick] (0,0) circle (2pt);
      \filldraw[white, thick] (0,0) circle (1.5pt);
    \end{tikzpicture} 
    \caption{$K_3$ is \Capra-convex for $\DoubleNorm{\cdot} = \Norm{\cdot}_2$}
    \label{fig:K3}
  \end{subfigure}
  \hfill
  \begin{subfigure}[b]{0.5\linewidth}
    \centering
    \begin{tikzpicture}[scale=0.8]
      \draw[gray, dashed] (0,0) circle (1.0); 
      \filldraw[cyan!30] (1/2, {sqrt(3)/2}) -- (1/2, {-sqrt(3)/2})  arc (-60:60:1) -- cycle ;
      \draw[cyan, thick] (1/2, {sqrt(3)/2}) -- (1/2, {-sqrt(3)/2})  arc (-60:60:1) -- cycle ;
      \draw[->] (-3,0)--(3,0) node[below right]{{$\primal_1$}};
      \draw[->] (0,-3)--(0,3) node[above left]{{$\primal_2$}};
      %
      \filldraw[mygreen, thick] (0,0) circle (3pt);
    \end{tikzpicture}
    \caption{$0 \notin \closedconvexhull\bp{\RadialProjection\np{\Cone_3}}$ 
      for $\DoubleNorm{\cdot} = \Norm{\cdot}_2$ } 
    \label{fig:co_n_K3}  
  \end{subfigure}
  \caption{The cones $\na{\Cone_i}_{i\in\na{1,2,3}}$
    in~\eqref{eq:cone_examples} (left column) 
    and the closed convex hull
    of their image by the radial projection
    $\RadialProjection$ in~\eqref{eq:normalization_mapping}
    defined with the source norm
    $\DoubleNorm{\cdot} = \Norm{\cdot}_2$
    (right column)}
  \label{fig:cones}%
\end{figure}

The example of $\Cone_2$ in Figure~\ref{fig:K2} reveals that a cone
needs not be convex to be \Capra-convex. 
Also, we notice that the \Capra-convexity of 
$\Cone_3$ in Figure~\ref{fig:K3} can be directly deduced
from Corollary~\ref{co:closed_convex_cones},
as $\Cone_3$ is a closed convex pointed cone.

\subsubsubsection{Examples with the nonrotund unit ball of the $\ell_\infty$ norm $\Norm{\cdot}_\infty$}

Let us now consider the source norm $\DoubleNorm{\cdot} = \Norm{\cdot}_\infty$.
We recall that, with this choice of source norm, the unit ball
$\BALL$ is not rotund, so that we rely mostly on Theorem~\ref{th:set_inter_spherezero} and
Corollary~\ref{co:closed_convex_cones} to characterize \Capra-convex sets.

We discuss the case of the three cones 
$\na{\Cone_i}_{i\in\na{1,2,3}}$
introduced in~\eqref{eq:cone_examples} 
and of the closed convex hull
of their image by the radial projection
$\RadialProjection$ in~\eqref{eq:normalization_mapping}, 
under the source norm $\DoubleNorm{\cdot} = \Norm{\cdot}_\infty$.
These sets are illustrated in Figure~\ref{fig:cones_l_infty}.

\begin{itemize}
\item For the cone $\Cone_1$ in Figure~\ref{fig:K1_l_infty},
  we have that 
  $\Cone_1 \cap \SPHEREzero_\infty = \na{(1, 0), (-1, 1), (-1, -1)}$,
  whereas the intersection of 
  $\closedconvexhull\bp{\RadialProjection\np{\Cone_1}}$ in
  Figure~\ref{fig:co_n_K1_l_infty}
  with $\SPHEREzero_\infty$ gives a larger set
  --- which contains for instance the origin $0$.
  We conclude from Theorem~\ref{th:set_inter_spherezero}
  that $\Cone_1$ is not \Capra-convex for the choice
  of source norm $\DoubleNorm{\cdot} = \Norm{\cdot}_\infty$.
  
\item For the cone $\Cone_2$ in Figure~\ref{fig:K2_l_infty},
  we have that 
  $\Cone_2 \cap \SPHEREzero_\infty = \na{(-1, 0), (-1, 1), (-1, -1)}$
  whereas the intersection of
  $\closedconvexhull\bp{\RadialProjection\np{\Cone_2}}
  =\convexhull\bp{\na{(-1, 1), (-1, -1)}}$ in Figure~\ref{fig:co_n_K2_l_infty} with 
  $\SPHEREzero_\infty$ gives the set 
  $\convexhull\bp{\na{(-1, 1), (-1, -1)}}$.
  We conclude from Theorem~\ref{th:set_inter_spherezero}
  that $\Cone_2$ is not \Capra-convex for the choice
  of source norm $\DoubleNorm{\cdot} = \Norm{\cdot}_\infty$.
  
\item For the cone $\Cone_3$ in Figure~\ref{fig:K3_l_infty},
  we have that 
  $\Cone_3 \cap \SPHEREzero_\infty = \closedconvexhull\bp{\RadialProjection\np{\Cone_3}} \cap \SPHEREzero_\infty$ 
  (see the representation of 
  $\closedconvexhull\bp{\RadialProjection\np{\Cone_3}}$
  in Figure~\ref{fig:co_n_K3_l_infty}).
  We conclude from Theorem~\ref{th:set_inter_spherezero}
  that $\Cone_3$ is \Capra-convex for the choice
  of source norm $\DoubleNorm{\cdot} = \Norm{\cdot}_\infty$.
\end{itemize}

\begin{figure}[htbp]
  \begin{subfigure}[b]{0.5\linewidth}
    \centering
    \begin{tikzpicture}[scale=0.8]
      \draw[gray, dashed] (-1,-1)--(-1,1) ;
      \draw[gray, dashed] (-1,1)--(1,1) ; 
      \draw[gray, dashed] (1,1)--(1,-1) ; 
      \draw[gray, dashed] (1,-1)--(-1,-1) ;
      \draw[->] (-3,0)--(3,0) node[below right]{{$\primal_1$}};
      \draw[->] (0,-3)--(0,3) node[above left]{{$\primal_2$}};
      \draw[blue, thick] (0,0)--({sqrt(8)},0);
      \draw[blue, thick] (0,0)--(-2,2);
      \draw[blue, thick] (0,0)--(-2,-2);
      \draw[black, thick] (0,0) circle (2pt);
      \filldraw[white, thick] (0,0) circle (1.5pt);
    \end{tikzpicture} 
    \caption{$K_1$ is not \Capra-convex for $\DoubleNorm{\cdot} = \Norm{\cdot}_\infty$}
    \label{fig:K1_l_infty} 
  \end{subfigure}
  \hfill
  \begin{subfigure}[b]{0.5\linewidth}
    \centering
    \begin{tikzpicture}[scale=0.8]
      \draw[gray, dashed] (-1,-1)--(-1,1) ;
      \draw[gray, dashed] (-1,1)--(1,1) ; 
      \draw[gray, dashed] (1,1)--(1,-1) ; 
      \draw[gray, dashed] (1,-1)--(-1,-1) ;
      \filldraw[cyan!30] (1,0) -- (-1,1) -- (-1, -1) -- cycle ;
      \draw[cyan, thick] (1,0) -- (-1,1) -- (-1, -1) -- cycle ;
      \draw[->] (-3,0)--(3,0) node[below right]{{$\primal_1$}};
      \draw[->] (0,-3)--(0,3) node[above left]{{$\primal_2$}};
    \end{tikzpicture}
    \caption{$\closedconvexhull\bp{\RadialProjection\np{\Cone_1}}$ 
      for $\DoubleNorm{\cdot} = \Norm{\cdot}_\infty$} 
    \label{fig:co_n_K1_l_infty} 
  \end{subfigure}
  
  \vskip\baselineskip
  
  \begin{subfigure}[b]{0.5\linewidth}
    \centering
    \begin{tikzpicture}[scale=0.8]
      \draw[gray, dashed] (-1,-1)--(-1,1) ;
      \draw[gray, dashed] (-1,1)--(1,1) ; 
      \draw[gray, dashed] (1,1)--(1,-1) ; 
      \draw[gray, dashed] (1,-1)--(-1,-1) ;
      \draw[->] (-3,0)--(3,0) node[below right]{{$\primal_1$}};
      \draw[->] (0,-3)--(0,3) node[above left]{{$\primal_2$}};
      \draw[blue, thick] (0,0)--({-sqrt(8)},0);
      \draw[blue, thick] (0,0)--(-2,2);
      \draw[blue, thick] (0,0)--(-2,-2);
      \draw[black, thick] (0,0) circle (2pt);
      \filldraw[white, thick] (0,0) circle (1.5pt);
    \end{tikzpicture} 
    \caption{$K_2$ is not \Capra-convex for $\DoubleNorm{\cdot} = \Norm{\cdot}_\infty$}
    \label{fig:K2_l_infty} 
  \end{subfigure}
  \hfill
  \begin{subfigure}[b]{0.5\linewidth}
    \centering
    \begin{tikzpicture}[scale=0.8]
      \draw[gray, dashed] (-1,-1)--(-1,1) ;
      \draw[gray, dashed] (-1,1)--(1,1) ; 
      \draw[gray, dashed] (1,1)--(1,-1) ; 
      \draw[gray, dashed] (1,-1)--(-1,-1) ;
      \draw[cyan, very thick] (-1,1) -- (-1,-1) -- cycle ;
      \draw[->] (-3,0)--(3,0) node[below right]{{$\primal_1$}};
      \draw[->] (0,-3)--(0,3) node[above left]{{$\primal_2$}};
    \end{tikzpicture}
    \caption{$\closedconvexhull\bp{\RadialProjection\np{\Cone_2}}$ 
      for $\DoubleNorm{\cdot} = \Norm{\cdot}_\infty$ } 
    \label{fig:co_n_K2_l_infty} 
  \end{subfigure}
  
  \vskip\baselineskip
  
  \begin{subfigure}[b]{0.5\linewidth}
    \centering
    \begin{tikzpicture}[scale=0.8]
      \filldraw[blue!30] (0,0) -- (2.5*1/2, {2.5*sqrt(3)/2}) -- (2.5*1/2, {-2.5*sqrt(3)/2}) -- cycle;
      \draw[gray, dashed] (-1,-1)--(-1,1) ;
      \draw[gray, dashed] (-1,1)--(1,1) ; 
      \draw[gray, dashed] (1,1)--(1,-1) ; 
      \draw[gray, dashed] (1,-1)--(-1,-1) ; 
      \draw[->] (-3,0)--(3,0) node[below right]{{$\primal_1$}};
      \draw[->] (0,-3)--(0,3) node[above left]{{$\primal_2$}};
      \draw[blue, thick] (0,0)--(2.5*1/2, {2.5*sqrt(3)/2});
      \draw[blue, thick] (0,0)--(2.5*1/2, {-2.5*sqrt(3)/2});
      \draw[black, thick] (0,0) circle (2pt);
      \filldraw[white, thick] (0,0) circle (1.5pt);
    \end{tikzpicture} 
    \caption{$K_3$ is \Capra-convex for $\DoubleNorm{\cdot} = \Norm{\cdot}_\infty$}
    \label{fig:K3_l_infty}
  \end{subfigure}
  \hfill
  \begin{subfigure}[b]{0.5\linewidth}
    \centering
    \begin{tikzpicture}[scale=0.8]
      \draw[gray, dashed] (-1,-1)--(-1,1) ;
      \draw[gray, dashed] (-1,1)--(1,1) ; 
      \draw[gray, dashed] (1,1)--(1,-1) ; 
      \draw[gray, dashed] (1,-1)--(-1,-1) ;
      \filldraw[cyan!30] ({1/sqrt(3)}, 1) -- ({1/sqrt(3)}, -1) 
      -- (1,-1) -- (1, 1) -- cycle ;
      \draw[cyan, thick] ({1/sqrt(3)}, 1) -- ({1/sqrt(3)}, -1) 
      -- (1,-1) -- (1, 1) -- cycle ;
      \draw[->] (-3,0)--(3,0) node[below right]{{$\primal_1$}};
      \draw[->] (0,-3)--(0,3) node[above left]{{$\primal_2$}};
      %
    \end{tikzpicture}
    \caption{$\closedconvexhull\bp{\RadialProjection\np{\Cone_3}}$ 
      for $\DoubleNorm{\cdot} = \Norm{\cdot}_\infty$} 
    \label{fig:co_n_K3_l_infty}  
  \end{subfigure}
  \caption{The cones $\na{\Cone_i}_{i\in\na{1,2,3}}$
    in~\eqref{eq:cone_examples} (left column) 
    and the closed convex hull
    of their image by the radial projection
    $\RadialProjection$ in~\eqref{eq:normalization_mapping}
    defined with the source norm
    $\DoubleNorm{\cdot} = \Norm{\cdot}_\infty$
    (right column)}
  \label{fig:cones_l_infty}%
\end{figure}

It is interesting to notice that the cone $\Cone_2$
is not \Capra-convex for the source norm 
$\DoubleNorm{\cdot} = \Norm{\cdot}_\infty$ but
is \Capra-convex for the source norm
$\DoubleNorm{\cdot} = \Norm{\cdot}_2$.
In particular, we observe that the conditions
in the right-hand side of~\eqref{eq:capra_convex_sets}
are still fulfilled when 
$\DoubleNorm{\cdot} = \Norm{\cdot}_\infty$,
since, in this case, we have that
$\Cone_2 \cap \na{0} = \closedconvexhull\bp{\RadialProjection\np{\Cone_2}} \cap \na{0} = \emptyset$ (see the representation of 
$\closedconvexhull\bp{\RadialProjection\np{\Cone_2}}$
in Figure~\ref{fig:co_n_K2_l_infty}).
This example highlights that the characterization~\eqref{eq:capra_convex_sets}
in Corollary~\ref{co:rotund_balls} is not sufficient
to identify a \Capra-convex set when the unit ball
induced by the source norm is not rotund,
as in the case of $\DoubleNorm{\cdot} = \Norm{\cdot}_\infty$.

Lastly, as with the source norm $\DoubleNorm{\cdot} = \Norm{\cdot}_2$,
the \Capra-convexity of 
$\Cone_3$ in Figure~\ref{fig:K3_l_infty} 
can be directly deduced
from Corollary~\ref{co:closed_convex_cones},
as $\Cone_3$ is a closed convex pointed cone.
Indeed, we recall that there is no assumption
on the source norm made in Corollary~\ref{co:closed_convex_cones}.

\subsection{The sublevel sets of \Capra-convex functions}
\label{subsec:Capra-convex_epigraphs_and_sublevel_sets_of_functions}

As in usual convexity, Definition~\ref{de:capra_convex_sets} of \Capra-convex
sets implies that the sublevel sets of \Capra-convex functions are \Capra-convex sets. 
Those sublevel sets appear in cardinality-constrained problems~\citep{Tillmann-Bienstock-Lodi-Schwartz:2024} 
through the \Capra-convex function~\(\lzero\) \cite{Chancelier-DeLara:2021_ECAPRA_JCA}.
This is however not the case for the epigraphs of \Capra-convex functions. 
We gather these two results in the following proposition.

We consider the mapping~\(\theta \colon \PRIMAL\times\RR \to \PRIMAL\times\RR\)
given by  \(\theta\np{\primal,t}=\bp{\RadialProjection\np{\primal},t}\),
for any \( \np{\primal,t} \in \PRIMAL\times\RR\). 
Then, introducing the coupling function
\(\star_\theta : \np{\PRIMAL\times\RR} \times \np{\PRIMAL\times\RR} \to \RR \)
given by \(\star_\theta \bp{\np{\primal, t}, \np{\dual, s}} 
= \proscal{\theta\np{\primal,t}}{\np{\dual, s}}\), for any
\(\bp{\np{\primal, t}, \np{\dual, s}} \in \np{\PRIMAL\times\RR} \times \np{\PRIMAL\times\RR}\).
We say that a set~\(\SubSet \subset \PRIMAL\times\RR \) is \emph{\(\star_\theta\)-convex} if \(
\Indicator{\SubSet} = \SFMbi{\Indicator{\SubSet}}{\star_\theta} \). 

\begin{proposition}
	\label{pr:sublevel_sets_epigraph_capra_convex_function}
	Let $\DoubleNorm{\cdot}$ be a source norm, and
	$\CouplingCapra$ be the corresponding \Capra-coupling
	as in Definition~\ref{de:Capra}. 
	Let \(\fonctionprimal \colon \PRIMAL \to \barRR\) be a \Capra-convex function, as in Definition~\ref{de:capra_convex_functions}.
	Then 
	\begin{enumerate}
		\item[(i)] 
		the sublevel sets~\(
		\fonctionprimal^{\leq t} = \defset{\primal \in \PRIMAL}{\fonctionprimal\np{\primal} \leq t}\)
		are \Capra-convex sets, for any \(t\in \RR\);
		
		\item[(ii)] the epigraph~\(\epigraph \fonctionprimal =
		\defset{\np{\primal,t}\in \PRIMAL \times \RR}{\fonctionprimal\np{\primal}
			\leq t}\) is a \(\star_\theta\)-convex set. 
	\end{enumerate}
\end{proposition}

\begin{proof}
	According to~\cite[Proposition~2.6]{Chancelier-DeLara:2021_ECAPRA_JCA}, there
	exists a proper \lsc\ convex function~\(\BigFonctionun \colon \PRIMAL \to \barRR\)
	such that \(\fonctionprimal = \BigFonctionun\circ \RadialProjection\).  Thus,
	\(\fonctionprimal\np{\primal} \leq t \iff \BigFonctionun\bp{\RadialProjection\np{\primal}} \leq
	t\), for any \(\primal \in \PRIMAL\) and \(t \in \RR\).
	
	Thus, elementary calculus rules on indicators functions yield that
	\begin{align*}
	\Indicator{\fonctionprimal^{\leq t}} &= 
	\Indicator{\RadialProjection^{-1}\np{\BigFonctionun^{\leq t}}}
	= \Indicator{\BigFonctionun^{\leq t}} \circ \RadialProjection
	\eqsepv \forall t \in \RR \eqfinv
	\\
	\Indicator{\epigraph \fonctionprimal} 
	&= \Indicator{\theta^{-1}\np{\epigraph \BigFonctionun}} 
	= \Indicator{\epigraph \BigFonctionun} \circ \theta
	\eqfinp
	\end{align*}
	Applying the reverse implication
	in~\cite[Proposition~2.6]{Chancelier-DeLara:2021_ECAPRA_JCA} ---
	that a set~\(\SubSet \subset \PRIMAL\times\RR \) is {\(\star_\theta\)-convex} if \(
	\Indicator{\SubSet} = \BigFonctiondeux \circ \theta\),
	where \(\BigFonctiondeux \colon \PRIMAL \to \barRR\) is a proper \lsc\ convex function --- 
	we conclude that
	\(\Indicator{\fonctionprimal^{\leq t}}\) is \Capra-convex, for any
	\(t \in \RR\). This proves~\((i)\).  Using the same argument,
	\(\Indicator{\epigraph \fonctionprimal}\) is a \(\star_\theta\)-convex set, as
	\(\Indicator{\epigraph \BigFonctionun}\) is a proper \lsc\ convex
	function. This proves~\((ii)\).
\end{proof}

\begin{remark}
	For a given source norm~\(\DoubleNorm{\cdot}\) and its associated radial
	projection~\(\RadialProjection \colon \PRIMAL \to \SPHEREzero\), there is no hope to
	rewrite the mapping~\(\theta \colon \PRIMAL \times \RR \to \PRIMAL \times \RR\) defined by
	\(\theta\np{\primal,t} = \bp{\RadialProjection\np{\primal},t}\) into the radial projection of
	an other norm.  Indeed, the mapping~\(\theta\) is not \(0\)-homogeneous, while
	radial projections are.
\end{remark}

\section{Conclusion}

We have given a definition for \Capra-convex sets,
together with results on the characterization of such sets, which are cones that need not be either convex nor closed.
Interestingly, we have obtained that the \Capra-convexity
of a set depends on the geometry of the unit ball 
induced by the source norm used to define the \Capra-coupling.
To complete these results, we have provided conditions for the conical hull of a set to be \Capra-convex.
These conditions allow to characterize sparse optimization problems
which could be addressed in the framework of \Capra-convexity.
In particular, we have showed that the conical hull of a closed spherically-convex set is \Capra-convex.
Lastly, we have provided several examples to illustrate the specificity of \Capra-convex sets.

\appendix
\section{Appendix: proofs}
\label{sec:proofs} 


\subsection{Generic lemmas}

We show a first characterization of \Capra-convex sets.

\begin{lemma}
	\label{le:capra_convex_set}%
	Let $\DoubleNorm{\cdot}$ be a source norm on $\PRIMAL$,
	$\CouplingCapra$ be the corresponding \Capra-coupling
	in~\eqref{eq:coupling_CAPRA},
	and $\RadialProjection$
	be the corresponding 
	radial projection in~\eqref{eq:normalization_mapping}.
	Let $\SubSet \subseteq \PRIMAL$ be a set.
	The following assertions are equivalent:
	\begin{itemize}
		\item[$(i)$] $\SubSet$  is \Capra-convex, 
		\item[$(ii)$]
		$\Indicator\SubSet
		=\Indicator{ \closedconvexhull\np{\RadialProjection\np{\SubSet}}} 
		\circ \RadialProjection$,  
		\item[$(iii)$] $\SubSet = \Converse{\RadialProjection}
		\Bp{\closedconvexhull\bp{\RadialProjection\np{\SubSet}}}$. 
	\end{itemize}
\end{lemma}

\begin{proof}
	To begin with, we show the following preliminary result:
	$\SFM{\Indicator\SubSet}{\CouplingCapra\CouplingCapra'}
	=\Indicator{ \closedconvexhull\bp{\RadialProjection\np{\SubSet}}}
	\circ \RadialProjection$.
	The coupling $\CouplingCapra$ in Definition~\ref{de:Capra}
	is one-sided linear (see~\cite[Definition~2.3]{Chancelier-DeLara:2021_ECAPRA_JCA}) and factorizes 
	with the radial projection $\RadialProjection \colon \PRIMAL \to \SPHEREzero$ in~\eqref{eq:normalization_mapping}.
	Therefore, we get that
	\begin{align*}
	\SFM{ \Indicator\SubSet }{\CouplingCapra\CouplingCapra'}
	&=
	\SFM{ \Indicator\SubSet }{\CouplingCapra\star'} \circ \RadialProjection
	\eqfinv
	\tag{as $\CouplingCapra$ is a one-sided linear coupling, see~\cite[Proposition 2.5]{Chancelier-DeLara:2021_ECAPRA_JCA}}
	\\
	&=
	\SFM{ \bp{\sigma_{ \RadialProjection\np{\SubSet}}} }{\star'} \circ \RadialProjection
	\eqfinv
	\tag{from $\SFM{ \Indicator\SubSet }{\CouplingCapra} = \sigma_{ \RadialProjection\np{\SubSet} }$, see~\cite[Proposition 2.5]{Chancelier-DeLara:2021_ECAPRA_JCA}}
	\\
	&=
	\SFM{ \bp{\sigma_{ \closedconvexhull\np{\RadialProjection\np{\SubSet}}}} }{\star'} \circ \RadialProjection
	\eqfinv
	\tag{see e.g.~\cite[Proposition 7.13]{Bauschke-Combettes:2017}}
	\\
	&=
	\SFM{ \bp{\Indicator{ \closedconvexhull\np{\RadialProjection\np{\SubSet}}}} }{\star\star'} \circ \RadialProjection
	\eqfinv
	\tag{see e.g.~\cite[Example 13.3]{Bauschke-Combettes:2017}}
	\\
	&=
	\Indicator{ \closedconvexhull\np{\RadialProjection\np{\SubSet}}} \circ \RadialProjection
	\eqfinp
	\tag{as $\closedconvexhull\np{\RadialProjection\np{\SubSet}}$ is closed and convex}
	\end{align*}
	
	Then, from the characterization of
	\Capra-convex functions in Definition~\ref{de:capra_convex_functions}
	and	\Capra-convex sets in Definition~\ref{de:capra_convex_sets},
	we obtain that the set $\SubSet$ is \Capra-convex \IFF
	$\Indicator\SubSet
	=\SFM{\Indicator\SubSet}{\CouplingCapra\CouplingCapra'}
	=\Indicator{ \closedconvexhull\np{\RadialProjection\np{\SubSet}}} 
	\circ \RadialProjection$.
	This proves $(i) \iff (ii)$. 
	The equivalence between $(ii)$ and $(iii)$
	is immediate as \( \Indicator{ \closedconvexhull\np{\RadialProjection\np{\SubSet}}} 
	\circ \RadialProjection = \Indicator{ \Converse{\RadialProjection}\bp{
\closedconvexhull\np{\RadialProjection\np{\SubSet}}}} \). 
\end{proof}

Additionally, we recall the following lemma, which is given as an exercise in~\cite[Exercise 3.48(a)]{Rockafellar-Wets:1998}.

\begin{lemma}
  \label{le:closeness}
	Let $\SubSet \subseteq \PRIMAL$ be a compact set such that $0 \notin \SubSet$.
	The set $\conicalhull(\SubSet) \cup \na{0}$ is closed.
\end{lemma}

\subsection{Proof of Theorem~\ref{th:set_inter_spherezero}}

\begin{proof}
 %
	We show the equivalence in~\eqref{eq:set_inter_spherezero} as two reverse implications.
  
  $(\implies)$ Let us assume that $\Cone$ is \Capra-convex. 
  We notice that, following Proposition~\ref{pr:cone},
  the set $\Cone$ is a cone. 
  Then, we have that
        \begin{align*}
          \Cone
          &=
            \Converse{\RadialProjection}
            \Bp{\closedconvexhull\bp{\RadialProjection\np{\Cone}}}
            \eqfinv
            \tag{by Item~(iii) in Lemma~\ref{le:capra_convex_set}}
          \\
		&=          
   \conicalhull\Bp{ \closedconvexhull\bp{\RadialProjection\np{\Cone}} \cap \SPHEREzero} \eqfinv
\tag{from~\eqref{eq:RadialProjection_inverse}}
               \intertext{from which we deduce that}
               \RadialProjection\np{\Cone}
          &= \conicalhull\Bp{ \closedconvexhull\bp{\RadialProjection\np{\Cone}} \cap \SPHEREzero} \cap \SPHEREzero
            =\closedconvexhull\bp{\RadialProjection\np{\Cone}} \cap \SPHEREzero
            \eqfinp
            \tag{from~\eqref{eq:cone_inter_spherezero}, as $\Cone$ is a cone}
        \end{align*}

	$(\impliedby)$ Let us assume that $\Cone$ is a cone and
        $\RadialProjection\np{\Cone} = \closedconvexhull\bp{\RadialProjection\np{\Cone}} \cap \SPHEREzero$.
	We have that
        \begin{align*}
          \Cone
          &=
            \Converse{\RadialProjection}\bp{\RadialProjection\np{\Cone}}
            \tag{by definition~\eqref{eq:normalization_mapping} 
            of the radial projection~$\RadialProjection$, and as $\Cone$ is a cone}
          \\
         &=
           \Converse{\RadialProjection}\Bp{
\closedconvexhull\bp{\RadialProjection\np{\Cone}} \cap \SPHEREzero }           
           \tag{by assumption}
           \\
         &=
           \Converse{\RadialProjection}\Bp{
\closedconvexhull\bp{\RadialProjection\np{\Cone}}}
    \tag{from~\eqref{eq:RadialProjection_inverse}}
           \eqfinp
        \end{align*}
        We conclude that that $\Cone$ is \Capra-convex
        by Item~(iii) in Lemma~\ref{le:capra_convex_set}.
\end{proof}

\subsection{Proof of Corollary~\ref{co:implication}}

\begin{proof}
	Let us suppose that
	the set $\Cone$ is \Capra-convex, and prove
	the implication $\np{\implies}$ in~\eqref{eq:capra_convex_sets}.
	
	$\bullet$ As shown in Proposition~\ref{pr:cone},
	$\Cone$ is a cone.
	
	$\bullet$ 
	Since $\Cone$ is a cone,
	we have that
	$\Cone \cup \na{0} = \conicalhull\np{\Cone\cap\SPHERE} \cup \na{0}$. Then, from Theorem~\ref{th:set_inter_spherezero},
	as $\Cone$ is \Capra-convex,
	we obtain
	$\Cone \cup \na{0} = \conicalhull\bp{\closedconvexhull\np{\Cone\cap\SPHEREzero} \cap \SPHERE} \cup \na{0}$.
	Lastly, as $\closedconvexhull\np{\Cone\cap\SPHEREzero} \cap \SPHERE$
	is a compact set which does not contain $0$,
	we conclude from Lemma~\ref{le:closeness}
	that $\Cone \cup \na{0}$ is closed.
	
	$\bullet$ The fact that 
	$\Cone \cap \na{0} = \closedconvexhull\bp{\RadialProjection\np{\Cone}} \cap \na{0}$   
	follows	from
	\begin{align*}
	0 \in \Cone
	&\iff \Indicator\Cone(0) = 0 \eqfinv \\
	&\iff \Indicator{ \closedconvexhull\bp{\RadialProjection\np{\Cone}}}(0) = 0 \eqfinv 
	\tag{from Lemma~\ref{le:capra_convex_set}, as $\RadialProjection(0) = 0$}\\
	&\iff 0 \in \closedconvexhull\bp{\RadialProjection\np{\Cone}}
	\eqfinp
	\end{align*}
	This ends the proof of the implication $\np{\implies}$.
\end{proof}

\subsection{Proof of Corollary~\ref{co:rotund_balls} (rotund norm balls)}

\begin{proof}
	We now prove that if the unit ball $\BALL$ of the source norm is rotund, then the
	reverse implication $\np{\impliedby}$ in~\eqref{eq:capra_convex_sets} holds. 
	
        Let us assume that the set $\Cone$ satisfies the three
        conditions in the right-hand side of the implication
        in~\eqref{eq:capra_convex_sets}, that is, 
        \[
              \begin{cases}
      \Cone \text{ is a cone,} \\
      \Cone \cup \na{0} \text{ is closed} \eqfinv \\
      \Cone \cap \na{0} =
      \closedconvexhull\bp{\RadialProjection\np{\Cone}} \cap \na{0}
      \eqfinp 
    \end{cases}
    \]
        We show that, under these assumptions,
        $\Cone \cap \SPHEREzero = \closedconvexhull\np{\Cone\cap\SPHEREzero} \cap \SPHEREzero$,
        hence that~\eqref{eq:set_inter_spherezero} is satisfied, from~\eqref{eq:cone_inter_spherezero}.
        
        The inclusion 
        $\Cone \cap \SPHEREzero \subseteq \closedconvexhull\np{\Cone\cap\SPHEREzero} \cap \SPHEREzero$
        is straightforward. We thus concentrate on the reverse 	inclusion.
        Let us take $\primal \in \closedconvexhull\np{\Cone\cap\SPHEREzero} \cap \SPHEREzero$.
        We consider two cases.
        
        $\bullet$ Let us assume that $\primal = 0$. 
        We deduce that $0 \in \closedconvexhull\np{\Cone\cap\SPHEREzero}
        =\closedconvexhull\bp{\RadialProjection\np{\Cone}}$, since
        $\Cone$ is a cone (see~\eqref{eq:cone_inter_spherezero}).
        Then, as $\Cone \cap \na{0} = \closedconvexhull\bp{\RadialProjection\np{\Cone}} \cap \na{0}$,
        we conclude that $0 \in \Cone$, and thus that
        $\primal = 0 \in \Cone \cap \SPHEREzero$.
        
        $\bullet$ We now turn to the case $\primal \neq 0$.
        We observe that
        \begin{equation*}
          \Cone \cap \SPHEREzero =
          \begin{cases}
            \bp{\Cone \cup \na{0}} \cap \SPHEREzero \eqsepv 
            &\text{ if } 0 \in \Cone \eqfinv
            \\
            \bp{\Cone \cup \na{0}} \cap \SPHERE \eqsepv 
            &\text{ if } 0 \notin \Cone \eqfinp
          \end{cases}
        \end{equation*}
        Since $\Cone \cup \na{0}$ is closed, we deduce that
        $\Cone \cap \SPHEREzero$ is closed, and thus compact (since included in the unit ball~$\BALL$).
        It follows that the convex hull of $\Cone \cap \SPHEREzero$ is compact
        (see e.g.~\cite[Corollary 2.30]{Rockafellar-Wets:1998}),
        and thus that
        $\closedconvexhull\np{\Cone\cap\SPHEREzero} 
        = \convexhull\np{\Cone\cap\SPHEREzero}$.
        
        
        We deduce that $\primal \in \convexhull\np{\Cone\cap\SPHEREzero}$,
        and therefore, from Carath\'eodory's Theorem (see e.g.~\cite[Theorem 2.29]{Rockafellar-Wets:1998}),
        that there exists $\na{\alpha_i}_{i\in\ic{1,{n}+1}} \in \nc{0,1}^{{n}+1}$,
        $\na{\primal_i}_{i\in\ic{1,{n}+1}} \in \np{\Cone\cap\SPHEREzero}^{{n}+1}$
        such that
        \begin{equation*}
          \primal = \sum_{i=1}^{{n}+1} \alpha_i\primal_i
          \eqsepv
          \text{and}
          \enspace
          \sum_{i=1}^{{n}+1} \alpha_i = 1 \eqfinp
        \end{equation*}
        Now, we observe that if for some
        $j \in \ic{1,{n}+1}$ we have $x_j = 0$ and $\alpha_j > 0$,
        as $x \in \SPHERE$,
        it implies that
        \begin{equation*}
          1 = \DoubleNorm{\primal}
          \leq 
          \sum_{i=1}^{{n}+1} \alpha_i\DoubleNorm{\primal_i}
          \leq \sum_{i\neq j, i=1}^{{n}+1} \alpha_i
          < 1 \eqfinv
        \end{equation*}
        which leads to a contradiction.
        We deduce that 
        $\na{\primal_i}_{i\in\ic{1,{n}+1}} \in \np{\Cone\cap\SPHERE}^{{n}+1}$,
        and thus that $x \in \convexhull\np{\Cone\cap\SPHERE}$.
        Finally, from~\cite[Corollary 16]{Chancelier-DeLara:2022_OSM_JCA},
        since the unit ball induced by $\DoubleNorm{\cdot}$ is rotund,
        we have that 
        $\convexhull\np{\Cone\cap\SPHERE} \cap \SPHERE = \Cone\cap\SPHERE$,
        so that $\primal \in \Cone\cap\SPHERE$.
        
        We conclude that, in both cases, $\primal \in \Cone\cap\SPHEREzero$,
        which proves that
        $\Cone \cap \SPHEREzero \supseteq \closedconvexhull\np{\Cone\cap\SPHEREzero} \cap \SPHEREzero$,
        and therefore that $\Cone \cap \SPHEREzero = \closedconvexhull\np{\Cone\cap\SPHEREzero} \cap \SPHEREzero$.
        
        Therefore, \eqref{eq:set_inter_spherezero} holds and 
        we deduce from Theorem~\ref{th:set_inter_spherezero} that
        the cone $\Cone$ is a \Capra-convex set.	
      \end{proof}

\subsection{Proof of Corollary~\ref{co:closed_convex_cones}(closed convex cones)}

\begin{proof}
	Let $\Cone \subseteq \PRIMAL$ be a closed convex cone.
	
	First, we prove~\((i)\)
	by showing that 
	$\closedconvexhull\np{\Cone\cap\SPHEREzero} = 
	\Cone\cap\BALL$.
	As $\Cone\cap\SPHEREzero \subset \Cone\cap\BALL$, we have that
	$\closedconvexhull\np{\Cone\cap\SPHEREzero} \subseteq \closedconvexhull\np{\Cone\cap\BALL}$.
	It follows that 
	$\closedconvexhull\np{\Cone\cap\SPHEREzero} \subseteq 
	\Cone\cap\BALL$ from the fact that $\closedconvexhull\np{\Cone\cap\BALL}=\Cone\cap\BALL$, since
	${\Cone\cap\BALL}$ is closed convex.
	
	To prove the reverse inclusion, 
	let us consider $\primal \in \Cone\cap\BALL$.
	By definition of the radial projection
	$\RadialProjection$ in~\eqref{eq:normalization_mapping},
	we have that 
	$\primal = \DoubleNorm{x} \RadialProjection(x) + (1 - \DoubleNorm{x}) 0$
	with $\na{\RadialProjection(x), 0} \subseteq \Cone\cap\SPHEREzero$
	as $\primal \in \Cone$ and $0 \in \Cone$,
	and $\DoubleNorm{x} \leq 1$, as $\primal \in \BALL$.
	We deduce that $x \in \convexhull\np{\Cone\cap\SPHEREzero}$,
	which proves the reverse inclusion
	$\closedconvexhull\np{\Cone\cap\SPHEREzero} \supseteq 
	\Cone\cap\BALL$.
	
	Now, we observe that
	$\closedconvexhull\np{\Cone\cap\SPHEREzero} \cap \SPHEREzero
	=\Cone\cap\BALL\cap\SPHEREzero = \Cone\cap\SPHEREzero$ which gives
	 that $\RadialProjection\np{\Cone}= \closedconvexhull\np{\RadialProjection\np{\Cone}} \cap \SPHEREzero$,
	    from~\eqref{eq:cone_inter_spherezero}, as $\Cone$ is a cone.
	    We obtain that $\Cone$ is \Capra-convex, from Theorem~\ref{th:set_inter_spherezero}.
	
	\medbreak
	
	Second, we assume that, moreover, the cone $\Cone$ is pointed, and we prove~\((ii)\).
	To proceed, we introduce the notation $\Cone' = \Cone \setminus \na{0}$ 
	and we show that $\Cone' \cap \SPHEREzero = \closedconvexhull\np{\Cone'\cap\SPHEREzero} \cap \SPHEREzero$.
	The inclusion $\Cone' \cap \SPHEREzero \subseteq \closedconvexhull\np{\Cone'\cap\SPHEREzero} \cap \SPHEREzero$
	is straightforward. 
	We thus concentrate on the reverse
	inclusion. 
	
	Let us take $\primal \in \closedconvexhull\np{\Cone'\cap\SPHEREzero} \cap \SPHEREzero$.
	As $\Cone'\cap\SPHEREzero = \Cone\cap\SPHERE$
	is a compact set, 
	so is its convex hull 
	(see e.g.~\cite[Corollary 2.30]{Rockafellar-Wets:1998}),
	and thus
	$\closedconvexhull\np{\Cone'\cap\SPHEREzero} 
	= \convexhull\np{\Cone\cap\SPHERE}$.
	We therefore have $\primal \in \convexhull\np{\Cone\cap\SPHERE} \cap \SPHEREzero
	\subset \convexhull\np{\Cone} \cap \SPHEREzero = \Cone \cap \SPHEREzero$, as $\Cone$ is convex.
	
	We now assume that $\primal = 0$, and show that
	it leads to a contradiction.
	As $x \in \convexhull\np{\Cone\cap\SPHERE}$,  
	we obtain , from Carath\'eodory's Theorem 
	(see e.g.~\cite[Theorem 2.29]{Rockafellar-Wets:1998}),
	that there exists $\na{\alpha_i}_{i\in\ic{1,{n}+1}} \in \nc{0,1}^{{n}+1}$,
	$\na{\primal_i}_{i\in\ic{1,{n}+1}} \in \np{\Cone\cap\SPHERE}^{{n}+1}$
	such that
	\begin{equation*}
	0 = \primal = \sum_{i=1}^{{n}+1} \alpha_i\primal_i
	\eqsepv
	\text{and}
	\enspace
	\sum_{i=1}^{{n}+1} \alpha_i = 1 \eqfinp
	\end{equation*}
	Thus, there necessarily exist $j\in\ic{1,{n}+1}$ such that $\alpha_j \neq 0$.
	By construction, $\primal_j \in \Cone$
	and 
	\begin{equation*}
	-\primal_j = \sum_{i\in\ic{1,{n}+1}, i\neq j} 
	\frac{\alpha_i}{\alpha_j}\primal_i
	\eqfinp
	\end{equation*}
	As $\Cone$ is a convex cone, we deduce 
	from the above expression
	that $-\primal_j \in \Cone$ (see e.g.~\cite[Proposition 6.3(i)]{Bauschke-Combettes:2017}).
	Then, since $\Cone$ is pointed, necessarily $\primal_j = 0$,
	which contradicts the fact that $\primal_j \in \SPHERE$.
	We conclude that $\primal \neq 0$, and therefore
	that $\primal \in \Cone \cap \SPHERE = \Cone' \cap \SPHEREzero$.
	
	This proves the reverse inclusion
	$\Cone' \cap \SPHEREzero \supseteq \closedconvexhull\np{\Cone'\cap\SPHEREzero} \cap \SPHEREzero$.
        Finally, we have obtained that
        $\Cone' \cap \SPHEREzero = \closedconvexhull\np{\Cone'\cap\SPHEREzero} \cap \SPHEREzero$ and
        the conclusion follows \eqref{eq:cone_inter_spherezero},
        as
        $\Cone'$ is a cone. This concludes the proof.
      \end{proof}

\subsection{Proof of Proposition~\ref{pr:coneX}}

\begin{proof}
  Let $\SubSet \subseteq \PRIMAL$
  be a compact set such that $0 \notin \convexhull(\SubSet)$,
  hence $0 \notin \SubSet$. 
  Following Lemma~\ref{le:closeness},
  the set $\conicalhull(\SubSet) \cup \na{0}$ is closed.
  \begin{itemize}
  \item Let us assume \((i)\).
		
    We start by proving that
    $0 \not\in \closedconvexhull\bp{\RadialProjection\np{\conicalhull(\SubSet)}}$
    by contradiction.  Assume that
    $0 \in \closedconvexhull\bp{\RadialProjection\np{\conicalhull(\SubSet)}}$.
    As $0 \notin \SubSet$ by assumption, and by definition~\eqref{eq:conicalhull} of the conical hull,
    	we have that $0\not\in \conicalhull(\SubSet)$, so that,
    	using~\eqref{eq:cone_inter_spherezero}, we deduce that  
    $0 \in \closedconvexhull\bp{\conicalhull(\SubSet) \cap \SPHERE}$ and thus
    $0 \in \convexhull\bp{\conicalhull(\SubSet) \cap \SPHERE}$ since
    $\conicalhull(\SubSet) \cap \SPHERE$ is compact (see e.g.~\cite[Corollary
    2.30]{Rockafellar-Wets:1998}).
    We obtain from Carath\'eodory's Theorem (see e.g.~\cite[Theorem
    2.29]{Rockafellar-Wets:1998}) that there exists
    $\na{\alpha_i}_{i\in\ic{1,{n}+1}} \in \nc{0,1}^{{n}+1}$,
    $\na{\primal_i}_{i\in\ic{1,{n}+1}} \in \bp{\conicalhull(\SubSet)\cap\SPHERE}^{{n}+1}$
    such that
    \begin{equation*}
      0 = \sum_{i=1}^{{n}+1} \alpha_i\primal_i
      \eqsepv
      \text{ and }
      \enspace
      \sum_{i=1}^{{n}+1} \alpha_i = 1 \eqfinp
    \end{equation*}
    By definition of the conical hull in~\eqref{eq:conicalhull},
    there exists 
    $\na{\primal'_i}_{i\in\ic{1,{n}+1}} \in \SubSet^{n+1}$
    such that $\primal_i = \RadialProjection(\primal'_i)$
    for $i \in \ic{1,n+1}$.
    Thus, introducing 
    $\theta = \sum_{i=1}^{{n}+1} \alpha_i / \Norm{\primal'_i} > 0$,
    we have
    \begin{equation*}
      0 = \sum_{i=1}^{{n}+1} 
      \frac{\alpha_i}{\theta\Norm{\primal'_i}} \primal'_i
      \eqsepv
      \text{ and }
      \enspace
      \sum_{i=1}^{{n}+1} \frac{\alpha_i}{\theta\Norm{\primal'_i}} = 1 \eqfinv
    \end{equation*}
    which implies that $0 \in \convexhull(\SubSet)$ and leads to a contradiction.
    
    We conclude that 
    $0 \notin \closedconvexhull\bp{\RadialProjection\np{\conicalhull(\SubSet)}}$
    and, therefore, that $\SubSet$
    is \Capra-convex, from the reciprocal implication
    in~\eqref{eq:capra_convex_sets} in Corollary~\ref{co:rotund_balls}
    in the case of a rotund unit norm ball.
		
  \item Now let us assume \((ii)\).

    Consider the cone $\Cone= \conicalhull(\SubSet) \cup \na{0} = \positivehull(\SubSet) $ which is convex as $\SubSet$ is convex.
    Moreover, we have that $0\notin \SubSet$, as $\SubSet$ is convex and $0 \notin \convexhull(\SubSet)$
    which, combined with Lemma~\ref{le:closeness} and the fact that $\SubSet$ is compact, gives that
    $\Cone$ is closed. Thus $\Cone$ is a closed convex cone.
    If we prove now that $\Cone$ is a pointed cone, then we will obtain from Corollary~\ref{co:closed_convex_cones}
    that $\conicalhull(\SubSet) = \Cone \backslash \na{0}$ is a \Capra-convex set.

    To conclude the proof, it remains to show that $\Cone$ is a pointed cone. We prove the result by contradiction.
    Assume that there exists $\primal \in \conicalhull(\SubSet) \cap \np{-\conicalhull(\SubSet)}$. Then,
    there exists $\lambda >0$ and $\beta >0$ and $x_1$, $x_2$ in $\SubSet$ such that
    $x = \lambda x_1$ and $x = - \mu x_2$. We therefore obtain that
    $0 = ( \lambda x_1 + \mu x_2)/(\lambda + \mu)$ and therefore that $0 \in\convexhull(\SubSet)$ which leads to a contradiction.
    We conclude that $\Cone \cap (-\Cone) = \na{0}$ and, therefore, that $\Cone$ is a pointed cone.
  \end{itemize}
  This concludes the proof.
\end{proof}

\newcommand{\noopsort}[1]{} \ifx\undefined\allcaps\def\allcaps#1{#1}\fi


\begin{thebibliography}{19}
\providecommand{\natexlab}[1]{#1}
\providecommand{\url}[1]{\texttt{#1}}
\expandafter\ifx\csname urlstyle\endcsname\relax
  \providecommand{\doi}[1]{doi: #1}\else
  \providecommand{\doi}{doi: \begingroup \urlstyle{rm}\Url}\fi

\bibitem[Bauschke and Combettes(2017)]{Bauschke-Combettes:2017}
H.~H. Bauschke and P.~L. Combettes.
\newblock \emph{Convex analysis and monotone operator theory in {H}ilbert
  spaces}.
\newblock CMS Books in Mathematics/Ouvrages de Math\'ematiques de la SMC.
  Springer-Verlag, New York, second edition, 2017.

\bibitem[Chancelier and {De Lara}(2021)]{Chancelier-DeLara:2021_ECAPRA_JCA}
J.-P. Chancelier and M.~{De Lara}.
\newblock Hidden convexity in the $\ell_0$ pseudonorm.
\newblock \emph{Journal of Convex Analysis}, 28\penalty0 (1):\penalty0
  203--236, 2021.

\bibitem[Chancelier and {De
  Lara}(2022{\natexlab{a}})]{Chancelier-DeLara:2022_CAPRA_OPTIMIZATION}
J.-P. Chancelier and M.~{De Lara}.
\newblock Constant along primal rays conjugacies and the $\ell_0$ pseudonorm.
\newblock \emph{Optimization}, 71\penalty0 (2):\penalty0 355--386,
  2022{\natexlab{a}}.
\newblock \doi{10.1080/02331934.2020.1822836}.

\bibitem[Chancelier and {De
  Lara}(2022{\natexlab{b}})]{Chancelier-DeLara:2022_SVVA}
J.-P. Chancelier and M.~{De Lara}.
\newblock Capra-convexity, convex factorization and variational formulations
  for the $\ell_0$ pseudonorm.
\newblock \emph{Set-Valued and Variational Analysis}, 30:\penalty0 597--619,
  2022{\natexlab{b}}.

\bibitem[Chancelier and {De Lara}(2023)]{Chancelier-DeLara:2022_OSM_JCA}
J.-P. Chancelier and M.~{De Lara}.
\newblock Orthant-strictly monotonic norms, generalized top-$k$ and $k$-support
  norms and the $\ell_0$ pseudonorm.
\newblock \emph{Journal of Convex Analysis}, 30\penalty0 (3):\penalty0
  743--769, 2023.

\bibitem[Fenchel(1983)]{Fenchel:1983}
W.~Fenchel.
\newblock Convexity through the ages.
\newblock In \emph{Convexity and its applications}, pages 120--130. Springer,
  1983.

\bibitem[Ferreira and Iusem(2014)]{Ferreira-Iusem-Nemeth:2014}
O.~P. Ferreira and S.~Z. Iusem, A. N.and~Németh.
\newblock Concepts and techniques of optimization on the sphere.
\newblock \emph{TOP}, 22\penalty0 (3):\penalty0 1148--1170, Oct 2014.
\newblock ISSN 1863-8279.
\newblock \doi{10.1007/s11750-014-0322-3}.

\bibitem[Guo and Peng(2021)]{Guo-Peng:2021}
Q.~Guo and Y.~Peng.
\newblock Spherically convex sets and spherically convex functions.
\newblock \emph{J. Convex Anal}, 28\penalty0 (1):\penalty0 103--122, 2021.

\bibitem[Hiriart-Urruty and
  Lemar{\'e}chal(2004)]{Hiriart-Urruty-Lemarechal:2004}
J.-B. Hiriart-Urruty and C.~Lemar{\'e}chal.
\newblock \emph{Fundamentals of convex analysis}.
\newblock Springer, 2004.

\bibitem[Jensen(1906)]{Jensen:1906}
J.~L. W.~V. Jensen.
\newblock Sur les fonctions convexes et les in{\'e}galit{\'e}s entre les
  valeurs moyennes.
\newblock \emph{Acta mathematica}, 30\penalty0 (1):\penalty0 175--193, 1906.

\bibitem[{Le Franc} et~al.(2024){Le Franc}, Chancelier, and {De
  Lara}]{LeFranc-Chancelier-DeLara:2022}
A.~{Le Franc}, J.-P. Chancelier, and M.~{De Lara}.
\newblock The capra-subdifferential of the $\ell_0$ pseudonorm.
\newblock \emph{Optimization}, 73\penalty0 (4):\penalty0 1229--1251, 2024.
\newblock \doi{10.1080/02331934.2022.2145172}.

\bibitem[{Mart\'{\i}nez-Legaz}(2005)]{Martinez-Legaz:2005}
J.~E. {Mart\'{\i}nez-Legaz}.
\newblock Generalized convex duality and its economic applications.
\newblock In N.~Hadjisavvas, S.~Koml\'osi, and S.~Schaible, editors,
  \emph{Handbook of Generalized Convexity and Generalized Monotonicity.
  Nonconvex Optimization and Its Applications}, volume~76, pages 237--292.
  Springer-Verlag, New York, 2005.

\bibitem[Moreau(1966-1967)]{Moreau:1966-1967}
J.~J. Moreau.
\newblock Fonctionnelles convexes.
\newblock \emph{S\'eminaire Jean Leray}, 2:\penalty0 1--108, 1966-1967.

\bibitem[Rakotomandimby et~al.(2025)Rakotomandimby, Chancelier, De~Lara, and
  Le~Franc]{Rakotomandimby:2025OL}
S.~Rakotomandimby, J.-P. Chancelier, M.~De~Lara, and A.~Le~Franc.
\newblock Subgradient selector in the generalized cutting plane method with an
  application to sparse optimization.
\newblock \emph{Optimization Letters}, 2025.
\newblock Accepted for publication.

\bibitem[Rockafellar(1974)]{Rockafellar:1974}
R.~T. Rockafellar.
\newblock \emph{Conjugate Duality and Optimization}.
\newblock CBMS-NSF Regional Conference Series in Applied Mathematics. Society
  for Industrial and Applied Mathematics, 1974.

\bibitem[Rockafellar and Wets(1998)]{Rockafellar-Wets:1998}
R.~T. Rockafellar and R.~J.-B. Wets.
\newblock \emph{Variational Analysis}.
\newblock Springer-Verlag, Berlin, 1998.

\bibitem[Simon and Rauhut(2013)]{Foucart-Rauhut:2013}
F.~Simon and H.~Rauhut.
\newblock \emph{A Mathematical Introduction to Compressive Sensing}.
\newblock {Birkh\"{a}user} New York, NY, 2013.
\newblock ISBN 978-0-8176-4948-7.
\newblock \doi{10.1007/978-0-8176-4948-7}.

\bibitem[Singer(1997)]{Singer:1997}
I.~Singer.
\newblock \emph{Abstract Convex Analysis}.
\newblock Canadian Mathematical Society Series of Monographs and Advanced
  Texts. John Wiley \& Sons, Inc., New York, 1997.
\newblock ISBN 0-471-16015-6.

\bibitem[Tillmann et~al.(2024)Tillmann, Bienstock, Lodi, and
  Schwartz]{Tillmann-Bienstock-Lodi-Schwartz:2024}
A.~M. Tillmann, D.~Bienstock, A.~Lodi, and A.~Schwartz.
\newblock Cardinality minimization, constraints, and regularization: a survey.
\newblock \emph{SIAM Review}, 66\penalty0 (3):\penalty0 403--477, 2024.

\end{thebibliography}
\end{document}